\documentclass{article}
\usepackage{amsmath}

\usepackage{amsfonts}
\usepackage{xcolor}
\usepackage{comment}
\newtheorem{theorem}{Theorem}

\newtheorem{corollary}[theorem]{Corollary}

\newtheorem{definition}[theorem]{Definition}

\newtheorem{lemma}[theorem]{Lemma}

\newtheorem{proposition}[theorem]{Proposition}
\newtheorem{remark}[theorem]{Remark}

\newenvironment{proof}[1][Proof]{\noindent\textbf{#1.} }{\ \rule{0.5em}{0.5em}}
\begin{document}

\title{The differential of self-consistent transfer operators and the local
convergence to equilibrium of mean field strongly coupled dynamical systems }
\author{Roberto Castorrini \\
Scuola Normale Superiore, roberto.castorrini@sns.it\\ \\
Stefano Galatolo \\
Dipartimento di Matematica, Università di Pisa, stefano.galatolo@unipi.it\\ \\
Matteo Tanzi \\
King's College London, matteo.tanzi@kcl.ac.uk }
\maketitle

\begin{abstract}
We consider the differential of a self-consistent transfer operator at a fixed point of the operator itself and show that its spectral properties can be used to establish a kind of local exponential convergence to equilibrium: probability measures near the fixed point converge
exponentially fast to the fixed point by the iteration of the transfer operator. This holds also in the strong
coupling case. We also show that for mean field coupled systems satisfying
uniformly a Lasota-Yorke inequality the differential does also. We present
examples of application of the general results to  self-consistent transfer operators based on deterministic expanding maps considered with different couplings,  outside the weak coupling regime.
\end{abstract}

\tableofcontents

\section{Introduction}

The dynamics of mean-field coupled maps exhibit a rich variety of emergent behaviors such as synchronization, clustering, partial ordering, turbulent/chaotic phases, and ergodicity breaking. These phenomena have been observed in natural settings and through numerical simulations since the end of the '80s (see for example \cite{PK,NK,Kan2,Kan,Just2,Just}  and \cite{BUMI} for a review). Yet, a comprehensive and general mathematical understanding remains elusive.

Most rigorous results have been obtained in the thermodynamic limit where the number of coupled systems tends to infinity. In this limit the state of the system is described by a probability measure whose dynamical evolution is given by a nonlinear mapping called a \textit{self-consistent transfer operator} (STO). The fixed points of an STO represent the equilibrium states for the system in the thermodynamic limit and provide information on the emergence of stationary behaviour for the system. It is therefore important to study existence,  stability, and the speed of convergence to these equilibrium states, as well as the changes of these equilibria under perturbations  (statistical stability, linear response, bifurcations). 

Recently, several results have appeared in the literature, mostly in the weak coupling case where the behavior of the self-consistent transfer operator is very close to linear. Existence and convergence to equilibrium states has been established for several types of uncoupled dynamics and interactions: see  {\cite{K} for }Tent maps; {\cite%
{SB} for }doubling maps and diffusive affine coupling; see {\cite{bal} for }%
smooth uniformly expanding maps and diffusive coupling; {\cite{ST} }smooth
uniformly expanding map sand smooth coupling; {\cite{BLS} coupled Anosov
maps; \cite{G} for random systems with additive noise; \cite{BK} for coupled intermittent maps. }

{Most of the results in the aforementioned works are typically derived from the spectral properties of the (linear) transfer operator associated with smooth, uniformly hyperbolic, uncoupled dynamics. However, recent studies have also explored the spectral features of these operators in cases of partial hyperbolicity (see \cite{CL}) or dynamics with discontinuities (e.g. \cite{BC}, \cite{BCC}), which are important to understand a variety of dynamics emerging in nature. Therefore, It is likely that the advances made in this and previous works will lead to equilibrium-type results even when the uncoupled maps do not exhibit smooth and uniform hyperbolicity.}

New rigorous approaches are on demand to treat the strong coupling regime where much less is understood mathematically, apart from few exceptions: {%
\cite{ST2} where synchronization results are presented; \cite{G} where
existence of equilibrium states is proved outside the weak coupling regime
and uniqueness results are proved in a kind of intermediate regime; \cite{BL}
where bifurcations and phase transitions are studied. }

It is expected that the stability of a fixed point can be deduced from the spectrum of the differential of the STO at the fixed point.   Instances of computation and use of this differential can be found in: \cite{wo}, where  the computation of the differential was used to support numerical evidence of the existence of non-hyperbolic large-scale dynamical structures in a mean-field coupled system, implying that the Gallavotti-Cohen chaotic hypothesis does not hold in that case;  in { \cite{BL} where the differential is used together with the Implicit Function Theorem  to give conditions for the continuation of the fixed point as the coupling strength increases, or  establishing the  existence of bifurcations leading to multiple invariant measures.}

However, in natural settings, when dealing with mean field coupled deterministic maps an STO is not Fr\'echet differentiable as a self-map of a given Banach space $( B, \|\cdot\|)$; it is instead differentiable if seen as mapping from a space with a strong norm, $( B_s, \|\cdot\|_s)$, to a space with a weak norm, $( B_w,\|\cdot\|_w)$, satisfying $ B_s\subset B_w$ and $\|\cdot\|_s\ge \|\cdot\|_{w}$\footnote{This loss of regularity is analogous to the loss of regularity encountered when studying linear response.}. This is a serious obstacle to establishing stability from knowledge of the spectrum of the differential.  In this paper we propose a systematic study of the differential of  STOs that overcomes this obstacle and provides sufficient conditions for the stability of equilibria.

\textbf{Overview and main results}: Let $(X,d)$ be a metric space representing the common phase space for the dynamics of the coupled subsystems. Let $PM(X)$ denote the set of Borel probability measures in $(X,d)$. The self-consistent transfer operator associated with the system is denoted by $\mathcal{L}_{\delta }:PM(X) \rightarrow PM(X)$, where $\delta\in \mathbb R$ is a parameter indicating the coupling strength ($\delta=0$ corresponds to the uncoupled case). An equilibrium state $h \in PM(X)$ is a fixed probability measure  of $\mathcal{L}_{\delta }$ ($\mathcal{L}_{\delta }h = h$).

We primarily study the behavior of $\mathcal{L}_{\delta }$ near $h$ outside the weak coupling regime. One main goal is to establish conditions under which $h$ is a stable fixed point (of $\mathcal{L}_{\delta }$ restricted to a space of suitably "regular" measures $B_{s}$) and the convergence to $h$ by iterating $\mathcal{L}_{\delta }$ is exponentially fast. This is interpreted as the exponential convergence to the equilibrium of the system.

Our main tool is the study of the differential $d\mathcal{L}_{\delta ,h}$ of $\mathcal{L}_{\delta }$ at $h$.
Let $V_s$ be the zero average space in $B_s$ (see \eqref{0avg}).
This differential can be defined as an operator $d\mathcal{L}_{\delta ,h}:V_{s}\rightarrow V_{s}$,  but in some important cases, such as the case where one is interested in mean field coupled deterministic maps, its convergence as a Fréchet differential occurs in a weaker topology. This complicates iterating the operator to prove exponential convergence to equilibrium. However, we establish that:

\begin{center}
  \emph{If one iterate of $d\mathcal{L}_{\delta ,h}$ is a strict contraction
on $V_{s}$, then some iterate of $\mathcal{L}_{\delta }$ strictly contracts a
neighborood of $h$ (see Proposition \ref{fond}).}
\end{center}

This criterion implies local exponential convergence to equilibrium for mean-field coupled systems outside the weak coupling regime, but as a corollary we also get a simple and general exponential convergence to equilibrium statement in the weak coupling regime (see Subsection \ref{weakc}).

Even though understanding the behavior of   $d\mathcal{L}_{\delta ,h}$ can be complex, we prove in Section \ref{sec:LYdiff} that if the systems we couple satisfy a certain Lasota-Yorke inequality, then $d\mathcal{L}_{\delta ,h}$ will also satisfy it.
This implies that if the functional analytic setting is chosen in a suitable way the operator has a discrete peripheral spectrum, and then the assumptions needed to apply our criterion for contraction (see Proposition \ref{fond} below)  can be verified by perturbative or reliable numerical methods.\footnote{We remark that the complicated form of the differential  seems to make the study of its spectral picture non trivial in the applications. However, because it satisfies a Lasota-Yorke inequality, it seems that in concrete examples one can test if an iterate of such an operator is a contraction on the zero-averages space by a computer-aided proof, for example, using methods similar to the ones introduced in \cite{GNS}.  }

In Section \ref{sec:3} we show an example of application of these results to concrete classes of systems. In Section \ref{ex1} we apply Proposition \ref{fond} to a class of expanding maps  with a simple deterministic coupling. We show exponential convergence to equilibrium even considering cases outside the weak coupling. 
The approach to the study of the differential here is perturbative.
We consider as uncoupled systems,  maps which are small perturbations of a map with linear branches. 
In this case, an expression 
for the differential
can be computed in a simple way
and we use this expression together 
with results on spectral stability for  quasi-compact operators  to 
derive the needed information on 
the spectrum of the differential of the STO.
Here, the fact that the differentials under study satisfy a certain uniform Lasota Yorke inequality, proved in Section \ref{sec:LYdiff}, is a key step.
In Section \ref{ex2} we consider a class of examples with a kind of sochastic coupling. In this case applying Proposition \ref{fond} we show that for small coupling strength there is only one equilibrium state, while if the coupling is strong enough, we have at least two equilibrium states having local exponential convergence to the equilibrium.

\section{Differential of the self-consistent transfer operator near a fixed point.}\label{S2}

In this section we set up the abstract framework in which we can prove our general main results. We follow  the approach of \cite{G} and \cite{BLS} where a self-consistent transfer operator is constructed starting from a family of Markov linear operators $L_{\delta, f}$ on suitable strong and weak spaces of measures, satisfying a certain list of natural assumptions.

Let us consider a metric space $X$, and the set of Borel probability measures $%
P(X)$. Let us consider the set $SM(X)$ of signed Borel measures on $%
X.$ This is a vector space, and let us consider normed vector subspaces $$%
(B_{ss},\| \cdot \|_{ss})\subseteq (B_{s},\|  \cdot \|_{s})\subseteq
(B_{w},\| \cdot\|_{w})\subseteq SM(X),$$ with $\|      ~\|_{ss}\geq \|\cdot\|_{s}\geq
\| \cdot\|_{w}.$

We denote the probability and the zero-average spaces as
\[
\begin{split}\label{0avg}
&\ P_{w}:=P(X)\cap B_{w},\qquad V_{w}:=\{\mu \in B_{w}|\mu (X)=0\}, \qquad V_{s}:=V_{w}\cap B_{s}.
\end{split}
\]
For $\tau,\tau' \in \{ss,s,w\}$, $\|\cdot\|_{\tau \to \tau'}$ will denote the operator norm from  $B_{\tau}$ to $B_{\tau'}$.\\
Let us consider $\delta \geq 0$, and  a family of Markov bounded
operators $L_{\delta ,f}:B_{w}\rightarrow B_{w}$ for any $f\in B_{w}.$ \footnote{Since we have to iterate operators of the form $L_{\delta ,f}$ for $f\in P_w$ and also take differentials of suitable nonlinear functions constructed starting from these operators, it would be sufficient to have $L_{\delta ,f}$ to be defined when $f$ ranges in a suitable small neighborhood of $P_w$. }
We
will assume that the family $L_{\delta ,f}$ preserves $B_{s}$ and $%
B_{ss}.$

Next, let us consider the \textit{Self-consistent Transfer Operator} (abbreviated by STO) $\mathcal{L}_{\delta }:P_{w}\rightarrow P_{w}$
associated to the family $L_{\delta ,f}$, defined as%
\begin{equation*}
\mathcal{L}_{\delta }(f):=L_{\delta ,f}(f).
\end{equation*}
In the above notation the parameter $\delta$  represents the strength of the coupling of the system. Its formal role will be related to the regularity of the family $L_{\delta ,f}$ as $f$ varies, as expressed in the following list of assumptions we set on  the family $L_{\delta ,f}$. 
\paragraph{Standing assumptions on $L_{\delta ,f}$.}

To study the behavior of $\mathcal{L}_{\delta }$ in a neighborhood of one of its fixed points $h$ and obtain convergence to equilibrium results, we will set list of assumptions which are natural in the context of mean field coupled systems and related operators. 
We will suppose that $h$ is regular and it is contained in $B_{ss}$; that all the operators are bounded and they satisfy a uniform sequential Lasota-Yorke estimate; that the operator $L_{\delta,h}$ associated to $h$ has convergence to equilibrium, and the family $L_{\delta,f}$ has some global Lipschitz regularity as $f$ varies.  Furthermore we will suppose some differentiable regularity for this family at the fixed point $f=h$.
\begin{description}
    
\item [(a)] (\textit{Fixed point}) There exists $h\in B_{ss}$ such that  $%
\mathcal{L}_{\delta }h=h.$
\item  [(b)](\textit{Boundedness}) There is $C\geq 0$ such that for each $ f\in B_{w}$ $\ $,%
\begin{equation}\label{e0}
\begin{split}
\|      &L_{\delta ,f}\|_{w\rightarrow w} \leq C,\\
\|     &L_{\delta ,f}\|_{s\rightarrow s} \leq C, \\
\|      &L_{\delta ,f}\|_{ss\rightarrow ss} \leq C.
\end{split}
\end{equation}

\item [(c)](\textit{Sequential Lasota Yorke inequalities}) The sequential composition of $L_{\delta ,f}~$satisfies a uniform
Lasota-Yorke inequality, that is: there are $\lambda <1$ and $A,B> 0$ such that
for each $n\in \mathbb{N}$ and each \marginpar{} $ \{f_{i}\}_{i=1}^n\in P_{w},g\in B_{s} 
\footnote{
The uniform assumption is stated for $f$ ranging in $P_{s},$ but, for the
purpose of this section, one could restrict to $f$ ranging in a small
neighborhood of $h$.}$%
\begin{equation} \label{e00}
\begin{split}
\|      L_{\delta ,f_{1}}\circ ...\circ L_{\delta ,f_{n}}(g)\|_{w} &\leq
B\|      g\|_{w},  \\ \\
\|      L_{\delta ,f_{1}}\circ ...\circ L_{\delta ,f_{n}}(g)\|_{s} &\leq A\lambda
^{n}\|      g\|_{s}+B\|      g\|_{w}.
\end{split}
\end{equation}

\item [(d)]\textit{(Convergence to equilibrium for $L_{\delta,h}$)} There exists $a_{n}\geq 0$ with $a_{n}\rightarrow 0$ such that,
for all $n\in \mathbb{N}$ and $\mu\in V_{s},$%
\begin{equation} \label{eq:Vs}
\|L_{\delta,h}^{n}(\mu)\|_{w}\leq a_{n}\|\mu\|_{s}. 
\end{equation}%

\item [(e)]\textit{(Regularity of the family of operators)} There exists $C_0>0$ such that, for each $ f,g\in B_{s}$,\marginpar{}
\begin{equation}
\|      L_{\delta ,f}(h)-L_{\delta ,g}(h)\|_{s}\leq \delta C_0\|      h\|_{ss}\|      f-g\|_{w},
\label{e1}
\end{equation}
 and there is $C_1>0$ such that, for each $f_{1},f_{2}\in B_{w},$~$g\in B_{s}$ \marginpar{}
\begin{equation}
\|      L_{\delta ,f_{1}}(g)-L_{\delta ,f_{2}}(g)\|_{w}\leq \delta
C_1\|      g\|_{s}\|      f_{1}-f_{2}\|_{w}.  \label{e1.5}
\end{equation}

\end{description}

We note that when $\delta$ is fixed,  in the above \eqref{e1} and \eqref{e1.5} its presence is not formally relevant.
However, it will be relevant in some cases where we will suppose that \eqref{e1} and \eqref{e1.5} hold uniformly as $\delta$ varies in a set having $0$ as an accumulation point.


In order to introduce later the differential of the self-consistent transfer operator, we now require that the function $g \to L_{\delta,g}(h)$ is differentiable in the \textit{Fr\'echet} sense at the fixed point $h$.

\begin{description}
    
\item [(f)](\textit{Differentiability of the family of operators at $h$}) Suppose there is a linear
bounded operator $$\partial L_{\delta ,h}(g ):V_w\rightarrow V_{s}$$
such that the following limit exists  \marginpar{}
\begin{equation}\label{e2}
\lim_{g\in V_w,\|      g \|_{w}\rightarrow 0}\frac{\|      L_{\delta ,h+g
}(h)-L_{\delta ,h}(h)-\partial L_{\delta ,h}(g )\|_{s}}{\|g \|_{w}}=0.
\end{equation}

\end{description}
\begin{remark}\label{normaL}
We remark that by assumption ${\bf (e)}$ we can estimate the norm of $\partial L_{\delta ,h}(g)$ as an operator from $V_w$ to $B_s$. Indeed
$$\|\partial L_{\delta ,h}(g)\|_{w\to s}=\lim_{\|g\|_w \to 0}\frac{\|\partial L_{\delta ,h}(g)\|_s}{\|g\|_w} $$
and
\begin{eqnarray}
\lim_{\|g\|_w \to 0}  \frac { \|\partial L_{\delta ,h}(g)\|_s}{\|g\|_w} &\leq & \lim_{\|g\|_w \to 0}\frac{\|      L_{\delta ,h+ g
}(h)-L_{\delta ,h}(h)\|_{s}}{\|g\|_w}
\\
&\leq &
 \delta C_0\|      h\|_{ss}.
\end{eqnarray}
  
\end{remark}


The following lemma provides sufficient conditions to define the differential of the operator $%
\mathcal{L}_{\delta }$.

\begin{lemma}[Fr\'echet differential of $\mathcal{L}_{\protect\delta }$]
\label{lemder} Suppose that  assumptions $\bf{(a)},\bf{(e)}$ and $\bf{(f)}$ are
satisfied. Then the differential operator $d\mathcal{L}_{\delta
,h}:V_s\rightarrow V_{s}$ of $\mathcal{L}_{\delta }$ computed at the
point $h,$\ defined by \marginpar{} 
\begin{equation}\label{defdiff}
\lim_{g \in V_s,\|g\|_{s}\rightarrow 0}\left\|\frac{\mathcal{L}_{\delta
}(h+g)-\mathcal{L}_{\delta }(h)-d\mathcal{L}_{\delta ,h}(g)}{\|g\|_{s}}\right\|_{w}=0
\end{equation}%
exists. Furthermore, we have that 
\begin{equation}
d\mathcal{L}_{\delta ,h}=L_{\delta ,h}+\partial L_{\delta ,h}  \label{e4}
\end{equation}
and $\ d\mathcal{L}_{\delta ,h}:$ $V_s\rightarrow V_{s}$ is a
bounded operator such that%
\begin{equation*}
\|d\mathcal{L}_{\delta ,h}\|_{s\rightarrow s}\leq \delta
C\|h\|_{ss}+C
\end{equation*}
where $C$ is the constant in \eqref{e0}.
\end{lemma}

\begin{proof}
Let $g\in V_{s},$ we can write%
\begin{equation*}\label{11}
\begin{split}
\mathcal{L}_{\delta }(h+g) &=L_{\delta,h+g}(h+g)  \\
&=L_{\delta ,h+g}(h)+L_{\delta,h+g}(g) \\
&=L_{\delta ,h}(h)+[L_{\delta,h+g}-L_{\delta ,h}](h)+L_{\delta ,h+g}(g).  
\end{split}
\end{equation*}
Hence%
\begin{equation*}
\begin{split}
&\frac{\|\mathcal{L}_{\delta }(h+g)-\mathcal{L}_{\delta }(h)-[L_{\delta
,h}(g)+\partial L_{\delta ,h}(g)]\|_{w}}{\|g\|_{s}} \\
&\le\frac{\|[L_{\delta ,(h+g)}-L_{\delta ,h}](h)-\partial L_{\delta
,h}(g)\|_{w}}{\|g\|_{s}} +\frac{\|L_{\delta ,(h+g)}(g)-L_{\delta ,h}(g)\|_{w}}{\|g\|_{s}}.
\end{split}
\end{equation*}
By assumption $\bf{(f)}$ $$\left\|\frac{1}{\|g\|_{s}}\left([L_{\delta ,(h+g)}-L_{\delta
,h}](h)-\partial L_{\delta ,h}(g)\right)\right\|_{w}\rightarrow 0$$ as $\|g\|_s \to 0$.  By $(\ref{e1.5})$, $%
L_{\delta ,h+g}$ $\rightarrow L_{\delta ,h}$ (in the stong-weak mixed norm)\ leading to \eqref{defdiff} and 
\eqref{e4}.\\
Finally, by Remark \ref{normaL} and assumption $\bf{(a)}$ we get the remaining part of the
statement.
\end{proof}

\begin{remark}\label{rmk:key}
It is important to remark that, while the operator $d\mathcal{L}_{\delta
,h}$ acts from the strong space $V_s$ to itself 
the convergence in \eqref{defdiff} is only in the weak topology. In particular, by the above proof, the fact that $h\in B_{ss}$ guarantees the convergence of $\frac{[L_{\delta ,(h+g)}-L_{\delta
,h}]}{\|g\|_{s}}(h)$ in $B_{s}$; however $\frac{L_{\delta ,(h+g)}(g)-L_{\delta ,h}(g)}{%
\|g\|_{s}}$ converges to $0$  
 only in the weak norm. This prevents us to use a direct iterative argument for the operators, needed to imply the main result. We overcome this complication later in the proof of Proposition \ref{fond}.
\end{remark}

\subsection{Local contraction of the self-consistent operator near a fixed point}
In this section we present our main result about the local stability and local exponential convergence to equilibrium for self-consistent transfer operators in a neighborhood of a fixed point $h$. 


\begin{definition}\label{defn:exp-dec}
(\textit{Local Strong Contraction}) We say that a self-consistent transfer operator has \textit{local strong contraction} (LOSC) to $h$ if, for some $ \epsilon>0$, there exist  $ \gamma , C> 0$,
such that, for each $\mu \in P_s$ such that $\|h-\mu \|_{s}\leq \epsilon$ and $n\in \mathbb{N}$
\begin{equation}\label{eq:WLED}
\|h-\mathcal{L}_{\delta }^{n}(\mu) \|_{s}\leq  Ce^{- \gamma n}\|h-\mu \|_{s}.
\end{equation}
\end{definition}
Clearly the local contraction implies a kind of local exponential convergence to equilibrium for the self-consistent operator:
 there exist  $\epsilon,\gamma ,C> 0$
such that, or each $\mu \in P_s$ such that $\|h-\mu \|_{s}\leq \epsilon$ and $n\in \mathbb{N}$
\begin{equation}\label{eq:WLEC}
\|h-\mathcal{L}_{\delta }^{n}(\mu) \|_{w}\leq Ce^{-\gamma n}\|h-\mu \|_{s}.
\end{equation}
The next proposition is our main  result, and it is a criteria to establish local strong contraction for the STO at a fixed point $h$.  It turns out that the STO is a local strong contraction if some iterate of the differential $d\mathcal{L}_{\delta
,h}$ is a contraction with respect to  the strong topology. 
Remarkably, this result holds independently of the strength of the coupling.

\begin{proposition} \label{fond} 
Under assumptions ${\bf (a)},{\bf (b)},{\bf (c)},{\bf (d)}, {\bf (e)}$ and ${\bf (f)}$, assume furthermore that
\begin{enumerate}
    
\item $\partial L_{\delta ,h}:V_w \to B_{ss}$ is bounded;
\footnote{Roughly speaking, in this assumption it is asked that the fixed point $h$ is even ``smoother" than the typical element in $B_{ss}$. It is written in terms of the differential in order to avoid to introduce a fourth space of measures.}
\item there are $n \in \mathbb{N}$ and a positive real $\lambda<1$ such that $\|d\mathcal{L}_{\delta
,h}^{n}(g)\|_{s}\leq \lambda \|g\|_{s}$ $\forall g \in V_s$.

\end{enumerate}
Then,  the system has local strong contraction to $h$.
\end{proposition}

\begin{proof}
We prove that under the above assumptions there exist $\epsilon>0$, $m \in \mathbb{N}$ and $\bar\lambda<1$ such that for each $g\in V_s$ with $\|g\|_s\leq \epsilon$
\begin{equation}\label{go1}
\|\mathcal{L}_{\delta }^{m}(h+g)-\mathcal{L}_{\delta }^{m}h\|_{s}\leq
\bar\lambda \|g\|_{s}. 
\end{equation}
The local  strong contraction directly follows from  \eqref{go1}. 
Since we cannot just iterate the differential (see Remark \ref{rmk:key}), in order to
prove $(\ref{go1})$ we set up an inductive reasoning on the number of
iterates. We suppose that at a fixed time $n$ we have for each $g\in  V_s$, when $\|g\|_s\to 0$
\begin{equation}\label{sop2}
\mathcal{L}_{\delta }^{n}(h+g)-\mathcal{L}^n_{\delta }(h)=d\mathcal{L}^n
_{\delta ,h}(g)+L_{g}^{(n)}(g)-L_{\delta ,h}^{n}(g)+o_{s\to s}(g)
\end{equation}%
where $\dfrac{\|o_{s\to s}(g)\|_s}{\|g\|_{s}}\to  0$ \  uniformly on $g$ as $%
\|g\|_{s}\rightarrow 0$ and $L_{g}^{(n)}$ is a sequential composition of
operators $L_{\delta ,f_{1}}\circ ...\circ L_{\delta ,f_{n}}$, where $%
f_{1},...,f_{n}$ depend on $\mathcal{L}_{\delta }^{1}(h+g),...,\mathcal{L}_{\delta }^{n}(h+g)$, such that 
\begin{equation}\label{eq:cruc}
\frac{\|L_{g}^{(n)}(g)-L_{%
\delta ,h}^{n}(g)\|_{w}}{\|g\|_{s}} \rightarrow 0 \qquad \text{as} \qquad \|g\|_{s}\rightarrow 0.
\end{equation}%
We shall prove that the same structure persists in the next iterate.\\
Note that, for ${g}\in V_s$, 
$$\mathcal{L}_{\delta }(h+{g})-\mathcal{L}_{\delta }(h)\in  V_s $$
and by the definition of $\partial L_{\delta,h}$ and Lemma \ref{lemder} we get 
\begin{equation}\label{eq:oneiteratediff}
\mathcal{L}_{\delta }(h+{g})-\mathcal{L}_{\delta }(h)=\partial
L_{\delta ,h}({g})+L_{\delta ,h+{g}}({g})+o_{s}({g}).
\end{equation}%
(which gives the base step of the induction).
Moreover, setting
\[
 g_n:= 
\mathcal{L}^n_{\delta }(h+{g})-\mathcal{L}^n_{\delta }(h) \in V_s\] we have by the inductive assumption that
\[ g_n= d\mathcal{L}_{h}( g)^{n}+L_{ g}^{(n)}( g)-L_{\delta
,h}^{n}( g)+o_{s\to s}(  g).
\]
 Recalling  that $\mathcal{L}^n_\delta h=h$ we can write (where in the last step we apply \eqref{eq:oneiteratediff})

\begin{equation}\label{eq:key}
\begin{split}
\mathcal{L}_{\delta }^{n+1}(h+g)-\mathcal{L}_{\delta }^{n+1}(h) &=\mathcal{L}%
_{\delta }(\mathcal{L}_{\delta }^{n}(h+g))-h \\
&=\mathcal{L}_{\delta }(h+g_n)-h \\
&=\partial L_{\delta ,h}( g_n) +L_{\delta ,h+ g_n}( g_n)+o_{s\to s}(g_n).
\end{split}
\end{equation}
We remark that by the inductive hypotesys, the strong continuity of $d\mathcal{L}_{h} $ and the standing assumption $(a)$ we have that $\|g_n\|_s=O(\|g\|_s)$, thus the term $o_{s}(g_n)$ is also of the type  $o_{s\to s}(g)$. 
We can then write 
\[\mathcal{L}_{\delta }^{n+1}(h+g)-\mathcal{L}_{\delta }^{n+1}(h)=\partial L_{\delta ,h}( g_n) +L_{\delta ,h+ g_n}( g_n)+o_{s\to s}(g).\]
Let us study the first term in the sum on the right hand side of the  equation above:
\begin{equation*}
\begin{split}
\partial L_{\delta ,h}( g_n)&=\partial L_{\delta ,h}(d\mathcal{L}_{\delta
,h}(g)^{n}+L_{g}^{(n)}(g)-L_{\delta ,h}^{n}(g)+o_{s\to s}(g)) \\
&=\partial L_{\delta ,h}(d\mathcal{L}_{\delta ,h}(g)^{n}) +\partial L_{\delta ,h}(L_{g}^{(n)}(g) -L_{\delta ,h}^{n}(g)+o_{s\to s}(g)).
\end{split}
\end{equation*}
Supposing that $g$ is nonzero, we write
\[
\partial L_{\delta,h}(L_{g}^{(n)}(g)-L_{\delta ,h}^{n}(g))=
\|g\|_{s}\partial L_{\delta,h}\bigg(L_{g}^{(n)}\bigg(\frac{g}{\|g\|_{s}}\bigg)-L_{\delta ,h}^{n}\bigg(\frac{g}{\|g\|_{s}}%
\bigg)\bigg),
\]
and by \eqref{eq:cruc} and \eqref{e2} (see Remark \ref{normaL}) 
\[
\frac{\|\partial L_{\delta,h}(L_{g}^{(n)}(g)-L_{\delta ,h}^{n}(g))\|_s}{\|g\|_s}\to 0, \qquad \text{as}\qquad \|g\|_s\to 0.
\]
Therefore, we have found
\begin{equation}\label{eq:J1}
\partial L_{\delta ,h}( g_n)= \partial L_{\delta ,h}(d\mathcal{L}_{\delta ,h}(g)^{n})+o_{s\to s}(g).
\end{equation}
The second summand in \eqref{eq:key} is estimated arguing similarly and using crucially that 
$\partial L_{\delta ,h}:V_w \to B_{ss}$ is bounded,  which implies $d\mathcal{L}_{\delta ,h}(g)^{n}-L_{\delta ,h}^{n}(g)\in
B_{ss}$  and there is some $\overline C\geq0$ such that $\|d\mathcal{L}_{\delta ,h}(g)^{n}-L_{\delta ,h}^{n}(g)\|_{ss}\leq \overline C \|g\|_w$ . 
\footnote{\label{n6} Indeed by \eqref{e4} $d\mathcal{L}_{\delta ,h}(g)=L_{\delta ,h}(g)+\partial L_{\delta
,h}(g)$ and  $d\mathcal{L}_{\delta ,h}^{{n}}(g)-L_{\delta ,h}^{{n}}(g)=\sum_{i=1}^{{n}^{2}-1}Q_{i}(g)$ where $Q_{i}$ are
operators containing all the possible $n-$factors product of the operators $L_{\delta ,h}$ and $\partial L_{\delta ,h }$ except $L_{\delta ,h}^{{n}}(g)$. Due to the regularization effect of $\partial L_{\delta ,h }$ and to the assumption $\bf{(a)}$ each one of the operators  $Q_{i}$ are bounded when seen as operators $V_w\to B_{ss}.$}

Hence,  by \eqref{e1}
\[
L_{\delta ,h+ g_n}(d\mathcal{L}_{\delta ,h}(g)^{n}-L_{\delta
,h}^{n}(g)+o_{s\to s}(g))=L_{\delta ,h}(d\mathcal{L}_{\delta ,h}(g)^{n}-L_{\delta
,h}^{n}(g)) +o_{s\to s}(g).
\]
It follows that%
\begin{equation}\label{eq:J2}
\begin{split}
&L_{\delta ,h+ g_n}( g_n)=\\
&=L_{\delta ,h+ g_n}(d\mathcal{L}_{\delta ,h}(g)^{n}+L_{g}^{(n)}(g)-L_{\delta ,h}^{n}(g)+o_{s\to s}(g)) \\
&=L_{\delta ,h+ g_n}(d\mathcal{L}_{\delta ,h}(g)^{n}-L_{\delta
,h}^{n}(g)+o_{s\to s}(g)) +L_{\delta ,h+ g_n}(L_{g}^{(n)}(g)) \\
&=L_{\delta ,h}(d\mathcal{L}_{\delta ,h}(g)^{n}-L_{\delta
,h}^{n}(g))+o_{s\to s}(g)) +L_{\delta ,h+ g_n}(L_{g}^{(n)}(g)).
\end{split}
\end{equation}
Setting
\begin{equation}\label{defLn}
L_{g}^{(n+1)}(g):=L_{\delta ,h+ g_n}(L_{g}^{(n)}(g))  
\end{equation}%
and using \eqref{eq:J1} and \eqref{eq:J2} into \eqref{eq:key}, we get%
\begin{equation*}
\begin{split}
&\mathcal{L}_{\delta }^{n+1}(h+g)-\mathcal{L}_{\delta }^{n+1}h \\
&=\partial L_{\delta ,h}(d\mathcal{L}_{\delta ,h}(g)^{n})+o_{s\to s}(g)\\
&+ L_{\delta ,h}(d\mathcal{L}_{\delta ,h}(g)^{n}-L_{\delta
,h}^{n}(g))+o_{s\to s}(g)) +L_{\delta ,h+ g_n}(L_{g}^{(n)}(g))\\
&=d\mathcal{L}_{\delta ,h}^{n+1}(g)-L_{\delta
,h}^{n+1}(g)+L_{g}^{(n+1)}(g)+o_{s\to s}(g).  
\end{split}
\end{equation*}
We thus proved \eqref{sop2}, by induction on $n\ge 0$.\\
We remark that the sequential composition $L_g^n$ satisfies \eqref{e00}.
We are now ready to prove \eqref{go1}. We will use Lemma \ref{losmem} in the appendix applied to the sequential composition of operators 
$L_i\circ ...\circ L_1 =L_g^{i}$
and $L_0=L_{\delta,h}$.
We remark that $(ML1)$ is given by \eqref{e00} and $(ML3)$ is given by \eqref{eq:Vs}. To satisfy $(ML2)$ we will apply this result uniformly  for   $g$ ranging in a small  
 $\epsilon$ neighborhood   of $h$ (see Note \ref{notaunif}).
Recalling that once fixed $m \in \mathbb{N}$  we have $%
\|(L_{\delta ,h}-L_{\delta,g_{j}})\|_{B_{s}\rightarrow B_{w}}= O( \epsilon )$ for each $j \le m$,
 it follows by Lemma \ref{losmem} that we can choose $\epsilon$ so small and $m$ large enough  in such a way that
\begin{equation*}
\|L^m_g\|_{s\to s}\leq \frac{1}{5}
\end{equation*} and
\begin{equation*}
\|L^m_{\delta,h}\|_{s\to s}\leq \frac{1}{5}.
\end{equation*}
Moreover, by the assumption on the differential, we can also suppose that
\begin{equation*}
\|d\mathcal{L}_{\delta ,h}^{m}\|_{s\to s }\leq \frac{1}{5}.
\end{equation*}%

By $(\ref{sop2})$ we thus get that, for each $g\in V_s$ with $%
\|g\|_{s}\leq \epsilon $,%
\begin{eqnarray*}
\|\mathcal{L}_{\delta }^{m}(h+g)-\mathcal{L}_{\delta }^{m}h\|_{s} &=&\|(d%
\mathcal{L}_{\delta ,h}(g))^{m}+L_{g}^{m}(g)-L_{\delta
,h}^{m}(g)+o_{s\to s}(g)\|_{s} \\
&\leq &\frac{3}{5}\|g\|_{s}+o_{s\to s}(g).
\end{eqnarray*}
Finally, we can find $\epsilon_1\leq \epsilon $ such that  if $g\in V_s$ is such that  $\|g\|_{s}\leq \epsilon _{1}$, the term $o_{s\to s}(g)$ is such that $
\|o_{s\to s}(g)\|_s\leq \frac{1}{5}\|g\|_{s}$. We conclude that%
\begin{equation*}
\|\mathcal{L}_{\delta }^{m}(h+g)-\mathcal{L}_{\delta }^{m}h\|_{s}\leq \frac{4%
}{5}\|g\|_{s},
\end{equation*}%
for each $g\in  V_s $ such that $\|g\|_{s}\leq \epsilon _{1}$, proving the
statement.
\end{proof}

\subsection{Convergence to equilibrium in the weak coupling }\label{weakc}
One interesting consequence of the results of the previous section is a relatively simple criteria for exponential  convergence to equilibrium in the weak coupling.
In the weak-coupling case it turns out that if the spectral radius of $L_{\delta ,h}|_{TQ_{h,s}}$ is less than $1$, then the spectral radius of $d\mathcal{L}_{\delta ,h}$ is also less than $1$, as it is shown in the following lemma.

\begin{lemma}\label{lemder3}
Let us suppose that there is a $\delta_0$ such  that  the family $L_{\delta,f}$ and the associated self-consistent transfer operator $\mathcal{L}_{\delta}$ satisfy the assumptions $\bf{(a)},...,\bf{(e)}$ and the assumption $(1)$ of Proposition \ref{fond} for  each $0\leq\delta \leq \delta_0$ with a fixed point $h_\delta$ depending on $\delta$ such that $\|h_\delta\|_{ss}$ is uniformly bounded. Suppose furthermore
$\exists n\in \mathbb{N}$ such that for  each $0\leq\delta \leq \delta_0$
\[
\ \|L_{\delta,h_\delta}^{n}\|_{V_s\rightarrow V_s}<1,
\]
 then there exists $\delta_1\leq \delta_0 $ such that for each $0<\delta \leq \delta_1$

\[
\|d\mathcal{L}_{\delta,h_{\delta}}^{n}\|_{V_s\rightarrow V_s}<1.
\]
\end{lemma}

\begin{proof} 
Let $0\leq\delta \leq \delta_0 $ and $g\in  V_s$, let us denote $d_{m}:=\|L_{\delta ,h_{\delta}}^{m}(g)-d\mathcal{L}_{\delta ,h_{\delta}}^{m}(g)\|_{s}$
then, by Remark \ref{normaL}
 $$d_{1}\leq \delta C\|g\|_{w}\|h_{\delta}\|_{ss}.$$ 
 Now let us consider further iterates
\begin{eqnarray*}
L_{\delta ,h_{\delta}}^{m}(g)-d\mathcal{L}_{\delta ,h_{\delta}}^{m}(g) &=&L_{\delta
,h_{\delta}}^{m}(g)-(L_{\delta ,h_{\delta}}(d\mathcal{L}_{\delta ,h_{\delta}}^{m-1}(g))+\partial
L_{\delta ,h_{\delta}}(d\mathcal{L}_{\delta ,h_{\delta}}^{m-1}(g))) \\
&=&L_{\delta ,h_{\delta}}^{m}(g)-L_{\delta ,h_{\delta}}(d\mathcal{L}_{\delta
,h_{\delta}}^{m-1}(g))-\partial L_{\delta ,h_{\delta}}(d\mathcal{L}_{\delta ,h_{\delta}}^{m-1}(g)) \\
&=&L_{\delta ,h_{\delta}}(L_{\delta ,h_{\delta}}^{m-1}(g)-d\mathcal{L}_{\delta
,h_{\delta}}^{m-1}(g))-\partial L_{\delta ,h_{\delta}}(d\mathcal{L}_{\delta ,h_{\delta}}^{m-1}(g)).
\end{eqnarray*}

Thus  $d_{m+1}\leq 
Cd_{m}+\delta C\|d\mathcal{L}_{\delta ,h_{\delta}}^{m-1}(g)\|_{w}\|h_{\delta}\|_{ss}$.
Since $\|h_{\delta}\|_{ss}$ is uniformly bounded and
$d\mathcal{L}_{\delta ,h}(g)$  is bounded as an operator from the weak space to itself,  it suffices to take $\delta $ small enough and $d_{m}$ is small as
wanted, for any fixed $m$, implying $\|d\mathcal{L}_{\delta,f}^{n}\|_{{V_s}\rightarrow V_s}<1$.
\end{proof}

Applying Proposition \ref{fond} to each $\mathcal{\delta}$ with  $0\leq \delta \leq \delta_1$ one directly gets the following.
\begin{corollary} Under the assumptions of Lemma \ref{lemder3} we have that for each $0\leq\delta \leq \delta_1 $ the self-consistent transfer operator $\mathcal{L}_\delta$ has local strong contraction to  its fixed point $h_\delta.$
    
\end{corollary}

\subsection{Lasota-Yorke inequality for the differential}\label{sec:LYdiff}
We conclude this section establishing some interesting facts about the spectral picture of the differential operator. 

Here, we prove that under the standing assumptions, the differential operator, regardless of the strength of the coupling, satisfies a Lasota-Yorke inequality. This implies quasi-compactness properties for the differential which can be very useful in order to get information on the stability of fixed points and of the spectral picture.

\begin{lemma}[Lasota-Yorke for the differential]\label{lem:LYdiff}
Under the assumptions of Proposition \ref{fond}, the operator
$d\mathcal{L}_{\delta ,h}^{n}:V_s\rightarrow V_{s}$ satisfies a Lasota-Yorke
inequality: there are $0\leq \tilde \lambda <1$ and $C_{4},C_{5} \geq 0$ such that for
each $n \in \mathbb{N}, g\in V_s$,%
\begin{equation*}
\|d\mathcal{L}_{\delta ,h}^{n}(g)\|_{s}\leq \tilde\lambda
^{n}C_4\|g\|_{s}+C_5\|g\|_{w}.
\end{equation*}
\end{lemma}

\begin{proof}
By $(\ref{e00})$ \ $\|L_{\delta ,f}^{n}(g)\|_{s}\leq \lambda
^{n}A\|g\|_{s}+B\|g\|_{w}.$ Let $\overline{n}$ be such that $\lambda
_{2}:=\lambda ^{\overline{n}}A<1$.

We are going to prove that there is  $B_{3}\geq 0$ such that 
\begin{equation}
\|d\mathcal{L}_{\delta ,h}^{\overline{n}}(g)\|_{s}\leq \lambda
_{2}\|g\|_{s}+B_{3}\|g\|_{w}.  \label{al}
\end{equation}

By \eqref{e4} $d\mathcal{L}_{\delta ,h}(g)=L_{\delta ,h}(g)+\partial L_{\delta
,h}(g)$. Thus $d\mathcal{L}_{\delta ,h}^{\overline{n}}(g)=L_{\delta ,h}^{%
\overline{n}}(g)+\sum_{i=1}^{\overline{n}^{2}-1}Q_{i}(g)$ where $Q_{i}$ are
operators containing all the possible remaining $\bar n-$factors product of the
operators $L_{\delta ,h}$ and $\partial L_{\delta ,h }$. Let $j_{i}\geq 0$ such that $Q_{i}=L_{\delta ,h}^{j_{i}}\partial
L_{\delta ,h } \hat{Q}_{i}$ \ where $\hat{Q}_{i}$ is a $%
[n-j_{i}-1]-$factors product of the operators $L_{\delta ,h}$ and $\partial
L_{\delta ,h,}$. By $(\ref{e00})$ and $(\ref{e2})$ (see Remark \ref{normaL}) we
have that there is a $K_{i}\geq 0\ $, depending on $\bar n,~j_{i}\ $and $%
\|h\|_{ss} $ but not on $g$ such that $\|\hat{Q}_{i}(g)\|_{w}\leq
K_{i}\|g\|_{w}$ and then by Remark \ref{normaL}
\begin{eqnarray}
\|\partial L_{\delta ,h}\cdot \hat{Q}_{i}(g)\|_{s} &\leq &\delta
C\|h\|_{ss}\|\hat{Q}_{i}(g)\|_{w}  \label{23} \\
&\leq &\delta C\|h\|_{ss}K_{i}\|g\|_{w}  \notag
\end{eqnarray}%
hence, for each such $i$ there is $\overline{K}_i\geq 0$  such that
$$\|Q_i(g)\|_s\leq \overline{K}_i \|g\|_w.$$
Hence by setting $\overline K:=\sum_{i=1}^{\overline{n}^{2}-1} \overline{K}_i$ we get
$\|\sum_{i=1}^{\overline{n}^{2}-1}Q_{i}(g)\|_{s}\leq \overline K\|g\|_{w}$.
Thus 
\begin{eqnarray}
\|d\mathcal{L}_{\delta ,h}^{\overline{n}}(g)\|_s &\leq & \|L_{\delta ,h}^{
\overline{n}}(g)\|_s+\|\sum_{i=1}^{\overline{n}^{2}-1}Q_{i}(g)\|_s
\\ 
&\leq &  A\lambda^{\bar n}
\|g\|_{s}+B \|g\|_{w}+ \overline K\|g\|_{w}
\end{eqnarray}
and $(\ref{al})$ \ is proved.
\end{proof}

\section{Application to strongly coupled expanding maps }\label{sec:3} 

 In this section, we apply Proposition \ref{fond} to some expanding maps outside the weak coupling regime. We will provide a class of examples where, thanks to the Lasota-Yorke inequality proved in Section \ref{sec:LYdiff} we can have estimates on the spectrum of the differential, to deduce strong local contraction in a neighborhood of fixed points.
 To get a priori estimates on the spectrum of the differential, we will consider maps which are small perturbations of expanding maps having linear branches.

Let us recall some basic properties of expanding maps.
Let $T:\mathbb{T}\to \mathbb{T}$ be a smooth $C^4$ uniformly expanding map on the circle, that is a map for which there exists $\sigma >1$ such that, for each $x\in \mathbb{T}$,
\[
|T'(x)|>\sigma.
\]

It is known that the transfer operator $T_*$ of $T$ acting on any signed measure with density $f\in L^1(\mathbb{T})$ returns a measure with density
\[
T_* f(x)=\sum_{y\in T^{-1}(x)}\dfrac{f(y)}{|T'(y)|}.
\]
In the following we will apply our general theory to self-consistent transfer operators constructed starting from these kind of operators, acting on stronger or weaker Sobolev spaces. In the following the choice of strong and weak spaces will be $B_w=L^1,B_s=W^{1,1},B_{ss}=W^{2,1.}$

Let's define the class of self-consistent transfer operators that we are going to study. First we define the contribution of the coupling to the expanding dynamics $T$. This will be done by defining a certain diffeomorphism   $\Phi_{\mu}:\mathbb{T}\rightarrow \mathbb{T}$ depending on a measure $\mu $ representing the global state of the system. We define a mean field  coupling by composing with  $\Phi_{\mu}$   as done more precisely in the following definition.

\begin{definition} \label{Def:Self-ConsOp} Given an uncoupled  map $T:\mathbb{T}\rightarrow \mathbb{T}$, a pairwise interaction function $ H\in C^4(\mathbb{T}\times\mathbb{T}, \mathbb{R}) $, a coupling strength parameter $\delta\geq 0 $, and any $\mu$, a probability measure on $\mathbb{T}$, define
\begin{itemize}
\item $\Phi_{\mu}:\mathbb{T}\rightarrow \mathbb{T}$,
\[
\Phi_{\mu}(x)=x+\delta\int_{\mathbb{T}}H(x,y)d\mu(y)\mod 1
\]
which is the mean-field coupling map when a (finite or infinite) system of coupled maps is in the state $\mu$;
\item $\mathcal L_\delta: PM(\mathbb{T})\rightarrow PM(\mathbb{T})$, the STO for the coupled system whose global state is represented by $\mu$ and defined as
\[
\mathcal L_\delta(\mu):=T_*([\Phi_{\mu}]_*(\mu)),
\]
where we denote $L_{\delta,
\mu}:= T_*\circ [\Phi_\mu]_*$  the push-forwards of $T\circ\Phi_\mu$.
\end{itemize}

\end{definition}

The type of mean field coupled dynamics and self-consistent transfer operator defined above  has been widely studied in the context of expanding and piecewise expanding maps, Anosov maps and random maps (see \cite{BUMI} for an extensive survey). It describes the thermodynamic limit of finite dimensional coupled systems as it is explained by two different approaches in \cite{ST} and \cite{G}.

The lemma below, whose simple proof is left to the reader, will be useful in what follows. In the following we denote $\partial_1^k H$ the $k$-th derivative of $H$ with respect to the first variable.
\begin{lemma}\label{Lem:BOundsGphi}
Given $\Phi_{\mu}$ defined as above
\begin{align*}
|\Phi_{\mu}'|_\infty\le 1+|\delta| |\partial_1H|_\infty,\quad
|\Phi_{\mu}''|_\infty\le |\delta| |\partial_1^2H|_\infty,\quad
|\Phi_{\mu}'''|_\infty\le |\delta| |\partial_1^3H|_\infty.
\end{align*}
Furthermore, for $|\delta|<|\partial_1H|_\infty^{-1}$, $\Phi_{\mu}$ is a diffeomorphism and for any $\bar \delta\in(0,|\partial_1H|_\infty^{-1})$ there is a constant $K$ depending on $|\partial_1H|_\infty$ and $\bar\delta$,  such that
\[
|(\Phi_{\mu}^{-1})'|_\infty\le 1+K|\delta|
\]
for every $\delta$ with $|\delta|<\bar\delta$.
\end{lemma}

If $\mu$ has a density $\phi$ with respect to Lebesgue, in the following we denote by  $\Lambda_\phi:=[\Phi_\phi]_*$ the coupling operator. One can give also an explicit expression for  $\Lambda_\phi$, i.e.
\[
    \Lambda_\phi f(x)=\dfrac{f}{\Phi_{\phi}'}(\Phi^{-1}_{\phi}(x)),  \; f\in L^1(\mathbb{T}), \; x\in \mathbb{T}.
\]

Since we are interested into the coupling of expanding maps outside the weak coupling regime, the following lemma gives a criteria based on the properties of $H$ and $T$ to ensure that with a certain non weak coupling strength $\delta$ the composition of the map with the coupling diffeomorphism is still an expanding map. 

\begin{lemma}\label{Lem:SmallDeltaUE}
Assume that $H$ is not constant and let $\rho_-:=  \min_{\mathbb{T}^2}\partial_1 H<0$. For each $\delta\in (0,(\sigma^{-1}-1)/\rho_-)$, the composition $T_\phi:=T\circ \Phi_\phi:\mathbb T\rightarrow\mathbb T$ is a uniformly expanding map for every $\phi\in P_s$, 
and furthermore there are $\sigma'>1$ and $K>0$ such that
\[
\inf_{x\in \mathbb{T}}|T_\phi'(x)|\ge \sigma'>1,\quad \|T_\phi\|_{C^3}\le K\quad\quad\forall \phi\in P_s.
\] 
\end{lemma}

\begin{proof}
For simplicity, let us use the notation $\mu(H_x)$ for $\int_{\mathbb{T}} H(x,y)d\mu(y)$ and similarly for the derivatives. For each $x\in \mathbb{T}$ we have
    \[
        |(T\circ \Phi_{\phi})'(x)|=|T'(x+\delta\mu(H_x))\cdot (1+\delta\mu(\partial_1 H_x))|\ge \sigma|1+\delta \int \partial_1 H(x,y)d\mu(y)|. 
    \]
    Under the assumption on $\delta$, $|\delta \int \partial_1 H(x,y)d\mu(y)|<1$, 
    \[
\left|1+\delta \int \partial_1 H(x,y)d\mu(y)\right|>\sigma^{-1}
    \]
    and $|T_\phi'(x)|>1$ for each $x$.
 The first statement and the inequality on the first derivative follows immediately. For the bound on the $\mathcal C^3$-norm, it is sufficient to note that
    \[
(T\circ \Phi_{\mu})'''(x)=T'''\circ \Phi_\mu (\Phi_\mu')^3+3T''\circ \Phi_\mu \Phi_\mu' \Phi_\mu''+T'\circ \Phi_\mu \Phi_\mu'''
    \]
    and apply Lemma \ref{Lem:BOundsGphi}.
\end{proof}

%

\subsection{An example with deterministically coupled maps}\label{ex1}
We now describe the class of maps where we mean to apply Proposition \ref{fond} and get strong contraction in a neighborhood of the fixed points. 

Let us consider $T:\mathbb T\rightarrow \mathbb T$ defined by $T(x)=kx \mod%
1$, with $k>1$. Let $\delta > 0$ and fix $H \in C^{4}(\mathbb T,\mathbb{R})$. For $\mu$ a probability on $\mathbb T^{1}$, let us define a special case of the couplings introduced in Definition \ref{Def:Self-ConsOp} where $H$ does not depend on the variable $x$:

\begin{equation*}
\Phi _{\delta ,\mu }(x)=x+ \delta \mu(H) \quad \mod 1
\end{equation*}%
where 
\[
\mu(H)=\int_{\mathbb T}H(y) d\mu(y).
\]
We remark that this is a diffeomorphism from $\mathbb{T}$ to $\mathbb{T}$, which is just a translation depending on $\mu$,
and it is of the form considered in \cite{BL}. {This simple case of coupling is chosen to simplify the computations. Rather than some attractive or repulsive coupling as usually considered in the applications, it can be thought as a collective forcing applied to all the  subsystems, which however depends on their collective state.
}
If $\mu$ has a density $h\in L^1(\mathbb{T})$, the STO associated to this coupling is then%
\begin{equation*}
\mathcal{L}_{\delta }h=T_* \Lambda_{\delta ,h }h
\end{equation*}%
where, according to the previous notations, we denote by $\Lambda_{\delta ,h }$ the transfer operator of the coupling $\Phi_{\delta,h}$. In this particular case, if $m$ is the Lebesgue measure, we have
\begin{equation}\label{eq:Lambda}
\Lambda_{\delta ,h }(f)(x)=\frac{f}{\Phi'_{\delta,h}}\circ \Phi^{-1}_{\delta,h}(x)=f(x-\delta \mu(H))=f\left(x-\delta\int_{\mathbb{T}} H h dm \right).
\end{equation}

Note that, for any coupling strength $\delta>0$, such a
system preserves the Lebesgue measure, hence the constant function $1$ corresponds to a fixed density for which we have global exponential convergence to equilibrium .

The transfer operator associated to the system is non-singular, we will then consider it as an operator from $W^{1,1}$ to $W^{1,1}$. For $\eta \in W^{1,1}$, the transfer operator of the map $T$ acts as \[
T_*\eta(x)=\sum_{i=0}^{k-1}\eta\left(\frac{x+i}{k}\right)k^{-1}.
\]
    Moreover, by \eqref{eq:Lambda}, for each $\eta\in W^{1,1}$, $(\Lambda_{\delta,h}\eta)'=\eta'$ and $\|\Lambda_{\delta,h}\eta\|_{L^1}=\|f\|_{L^1}$. Therefore, it is immediate to check that, for any $\delta>0, \eta\in W^{1,1}, h\in W^{1,1}$, $T_*$ contracts the $L^1$ norm of the weak derivative
 \[
\|(T_*\Lambda_{\delta,h} \eta)'\|_{L^1}\le k^{-1}\|\eta'\|_{L^1}.
 \]
 It follows that, for each $\eta \in W^{1,1}$
which is a probability density in $\mathbb{T}$, we
have for each $n\in \mathbb{N}$%
\begin{equation*}
\|1-\mathcal{L}_{\delta}^{n}\eta \|_{W^{1,1}}\leq C k^{-n}
\end{equation*}
which gives exponential convergence to equilibrium.
\footnote{Note that in this particular case the convergence to equilibrium is global, and not just in a neighborhood of the fixed point $1$.}

In the following proposition we shall prove, applying Proposition \ref{fond} that any small perturbation of the uncoupled map still gives an STO that enjoys LOSC even in the strong coupling. 

More precisely, for $\epsilon_0 >0 $, we consider any $\{T_\epsilon\}_{0<\epsilon<\epsilon_0} \in \mathcal{C}^4(\mathbb{T},\mathbb{T})$ such that 
\begin{itemize}
\item $\|T-T_\epsilon\|_{\mathcal{C}^4} \le \epsilon$,
\item there is $\sigma'>0$ such that $T'_\epsilon(x) \geq\sigma'>1$ for each $x\in \mathbb{T}$.
\end{itemize}
Clearly, the STO in this case is just
\[
\mathcal{L}_{\delta }^{(\epsilon)}h=(T_\epsilon)_* \Lambda_{\delta ,h }h.
\]
Under these assumptions we have the following.

\begin{proposition}\label{prop:exstrong} 
For each $\delta>0$ there exists $\epsilon_0>0$ such that, for each $\epsilon\in [0,\epsilon_0]$ and for each $h_\epsilon \in W^{3,1}(\mathbb{T})$ which is a fixed point of  the self-consistent transfer operator $\mathcal{L}^{(\epsilon)}_{\delta}$, we have local exponential contraction in a neighborhood of $h_\epsilon$: there are $C,\varrho ,\lambda >0$ such that, for each $\epsilon \in [0,\epsilon_0]$ and for each $\eta \in W^{1,1}$, 
which is a probability density with $\|h_\epsilon-\eta \|_{W^{1,1}}\leq \varrho$, we
have  
\begin{equation*}
\|h_\epsilon-[\mathcal
{L}^{(\epsilon)}_{\delta }]^{n}\eta \|_{W^{1,1}}\leq Ce^{-\lambda n}\quad\quad n\in \mathbb{N}.
\end{equation*}
\end{proposition}


\begin{proof}
In order to apply our main result, Proposition \ref{fond}, to the perturbed system, we first need to check assumptions ${\bf (a)},..., {\bf (f)}$.  We show that for $\epsilon_0$ small enough the family of transfer operators $L^{(\epsilon)}_{\delta, h}=T_{(\epsilon)*}\circ \Lambda _{\delta ,h }$ acting on the spaces $%
B_{w}=L^{1},B_{s}=W^{1,1},B_{ss}=W^{2,1}$ satisfies the standing assumptions ${\bf (a)},..., {\bf (e)}$ of Section \ref{S2}.
Note that since the coupling is just a translation, 
for each $x\in \mathbb T$, $h\in P_s$, $0\leq\epsilon \leq \epsilon_0$
\[
|(T_\epsilon \circ \Phi_{\delta,h})'(x)|=|T_{\epsilon}'(x+\delta\mu(H))|\ge \sigma'>1, 
\]
and $\|T_\epsilon \circ \Phi_{\delta,h}\|_{C^4}=\|T_\epsilon \|_{C^4}$
so that  $T_\epsilon\circ \Phi_{\delta,h}$ is a uniformly expanding map, and satisfies a uniform Lasota-Yorke inequality independently on 
$\delta$. Thus assumptions ${\bf(b),(c), (d)}$ are 
verified. Under the current assumptions the existence of a fixed point $ h_\epsilon \in B_{ss}$ 
(assumption ${\bf(a)}$) and the regularity of the family of operators (assumption ${\bf(e)}$) are proved in  \cite[Section 7]{G}, in a more general setting.\footnote{Nevertheless, in \cite{G} the convergence to equilibrium is shown in a weak coupling situation.}

It remains to prove ${\bf (f)}$. Since all the maps involved are $C^4$ we have $h_\epsilon \in W^{3,1}(\mathbb T)$. Let us fix some $\eta \in
L^{1}(\mathbb T)$.\\
We show the existence of the operator $\partial L^{(\epsilon)}_{\delta, h}$ computing the Gateaux differential at $h_\epsilon$ with increment $\eta$, and showing that the formula is uniform as $\eta$ varies in the unit ball of the weak space, which in this case is $V_w$. Since for each $t\in \mathbb R$, $\Lambda_{\delta ,h_\epsilon +t \eta }(h_\epsilon )=\Lambda_{\delta ,h_\epsilon }(\Lambda_{\delta
,t \eta }(h_\epsilon))$, we have 
\begin{equation*}
\underset{%
t \rightarrow 0}{\lim }[T_\epsilon]_*\left(\frac{\Lambda_{\delta ,h_\epsilon +t \eta
}(h_\epsilon)-\Lambda_{\delta,h_\epsilon }(h_\epsilon)}{t}\right)=\underset{%
t \rightarrow 0}{\lim }[T_\epsilon]_*\left[\Lambda_{\delta ,h_\epsilon}\left(\frac{\Lambda_{\delta
,t \eta }(h_\epsilon )-h_\epsilon }{t}\right)\right].
\end{equation*}
Next, the family of functions 
\[
f_t:=\frac{\Lambda_{\delta
,t \eta }(h_\epsilon(x) )-h_\epsilon(x) }{t}=\frac{h_\epsilon(x-\delta t \int H \eta dm)-h_\epsilon(x)}{t}
\]
converges $\mu$-almost everywhere to $h_\epsilon' \delta\int H \eta dm$, where $h_\epsilon'$ is the weak derivative of $h_\epsilon$. 
Since $h_\epsilon \in W^{3,1}$, by the dominated convergence theorem we have that this limit also converges in $W^{1,1}$. Hence
\[
\underset{t \rightarrow 0}{\lim }[T_\epsilon]_*\left[\frac{\Lambda_{\delta ,h_\epsilon
+t \eta }(h_\epsilon )-\Lambda_{\delta ,h_\epsilon }(h_\epsilon )}{\epsilon }\right]=\delta\left(\int_{\mathbb T}\eta H  dm\right) ~[T_\epsilon]_*[\Lambda_{\delta ,h_\epsilon }(h_\epsilon ^{\prime })]
\] with convergence in $W^{1,1}$.
The verification of ${\bf (f)}$ follows then from the fact that the above computation is uniform as $\eta$ varies in the unit ball of 
 $V_w$ (since it depends on $\eta$ only through its integral against $H$). 
Applying Lemma \ref{lemder} we have now a formula for the differential of $L^{(\epsilon)}_{\delta,h_\epsilon}$ for each $\epsilon \in [0,\epsilon_0]$:
\begin{equation}\label{eq:dL1}
d\mathcal{L}^{(\epsilon)}_{\delta,h_\epsilon }(\eta )=[T_\epsilon]_*\Lambda_{\delta ,h_\epsilon }\eta +\delta  \left(\int_{\mathbb T} \eta H
dm\right)~[T_\epsilon]_*[\Lambda_{\delta ,h_\epsilon }(h_\epsilon ^{\prime })].
\end{equation}
Now we need to check the other two assumptions 1) and 2) of Proposition \ref{fond}.
 Since $h_\epsilon\in W^{1,3}$,  the operator $\partial L^{(\epsilon)}_{\delta,h_\epsilon}$ defined by $$\partial L^{(\epsilon)}_{\delta,h_\epsilon}(h)=\delta  \left(\int_{\mathbb T} \eta H
dm\right)~[T_\epsilon]_*[\Lambda_{\delta ,h_\epsilon }(h_\epsilon ^{\prime })]$$ is obviously bounded as an operator $V_w\to B_{ss}$ proving hypothesis 1).

Finally, to prove hypothesis 2) we show it for $\epsilon=0$ and then we use a result of \cite{KL} to see that the spectrum of the perturbed differential is close to that of the unperturbed differential, from which the conclusion will follow.

The Lebesgue measure $m$ is a fixed point for\ $\mathcal{L}^{0}_{\delta }$ and
its density $h:=h_0=1$ is obviously contained in $W^{3,1}.$ If we compute the differential at the Lebesgue measure we have $h ^{\prime }=0$ everywhere, so that $d%
\mathcal{L}^{0}_{\delta,h}=T_*\Lambda_{\delta ,h }$ \ and we conclude that the operator $d%
\mathcal{L}_{\delta,h}$ restricted to $V_{s}$ has
spectral radius strictly smaller than one, since this is the transfer operator associated to a translation of the linear map $T_0$.\\
When $\epsilon \neq 0$ the invariant density is not constant, but close to it.
We can thus use the result of \cite{KL} to get the stability of the spectrum under the perturbation leading $\mathcal{L}^{0}_{\delta,h}$ to $\mathcal{L}^{\epsilon}_{\delta,h}$ . Since, there is a compact immersion of $B_s$ into $B_w$, to apply the main result of \cite{KL}  we need to find $\epsilon_0>0$ such that we can verify that the family of operators $\mathcal{L}^{\epsilon}_{\delta,h}$,  satisfy:
 \begin{itemize}
     \item[(i)] (Uniform Lasota-Yorke inequality) There are $C_1,C_2,M>0,$ and $\sigma<1$ such that, for each $\epsilon \in [0, \epsilon_0)$ and $n\in \mathbb{N}$,
     \[
    \|[d\mathcal{L}^{(\epsilon)}_{\delta,h_\epsilon}]^n \eta\|_{L^1}\le C_1 M^n\|\eta\|_{L^1}
     \]
     and
     \[
    \|[d\mathcal{L}^{(\epsilon)}_{\delta,h_\epsilon}]^n \eta\|_{W^{1,1,}} \le C_2\sigma^{n}\|\eta\|_{W^{1,1}}+C_1 M^n \|\eta\|_{L^1}.
     \]
     \item[(ii)] (Weak closeness) There exists $C_3>0$ such that, for each $\epsilon \in (0,\epsilon_0)$,
     \[
     \sup_{\eta\in W^{1,1}: \|\eta\|_{W^{1,1}}\le 1} \|d\mathcal{L}^{(\epsilon)}_{\delta,h_\epsilon}\eta-d\mathcal{L}^{(0)}_{\delta,h}\eta\|_{L^1}\le C_3 \epsilon.
     \]
 \end{itemize}

We can check in the proof of Lemma \ref{lem:LYdiff} that $d\mathcal{L}^{(\epsilon)}_{\delta,h_\epsilon}$ satisfies a Lasota-Yorke inequality in $W^{1,1}$ and $L^1$ whose coefficients depend on the ones of the Lasota Yorke inequality of $L^{\epsilon}_{\delta,h}$ and to the norm $\|\partial L^\epsilon_{\delta,h}\|_{w\to s}.$ By Remark \ref{normaL} this in turn depends on $\|h_\epsilon\|_{ss}$, which is uniformly bounded as $\epsilon$ ranges in some interval $[0,\epsilon_0)$, provided $\epsilon_0 $ is small enough.

By the fact that $T$ and $T_\epsilon$ are $\epsilon$ close in $\mathcal{C}^4$,
the coefficients of the Lasota-Yorke inequalities of the operators
$L^{\epsilon}_{\delta,h}$ are uniformly bounded, again  as $\epsilon$ ranges in some interval $[0,\epsilon_0)$, provided $\epsilon_0 $ is small enough.
Hence we get that $d\mathcal{L}^{(\epsilon)}_{\delta,h_\epsilon}$ satisfies a Lasota-Yorke inequality in $W^{1,1}$ and $L^1$ which is uniform as $0 \leq \epsilon<\epsilon_0$ provided $\epsilon_0 $ is small enough.

Let us show $(ii)$. 
Before showing this we note that for any choice of the fixed probability densities $h_{\epsilon}$ of the operators $\cal{L}^{\epsilon}_{\delta}$ we have \begin{equation} \label{18}\|h_{\epsilon}-1\|_{W^{1,1}}\le C\epsilon.
\end{equation}
This is because $h_\epsilon$ is the unique fixed probability measure in $W^{3,1}$ of the expanding map $T_{\epsilon}\circ \Phi_{\delta,h_\epsilon}$ where $\Phi_{\delta,h_\epsilon}$ is a translation, therefore \eqref{18} follows from \cite[Proposition 48]{Gal}. Then it is the unique invariant probability density of a map which is  $\epsilon -$close in the $C^r$ topology to an expanding map having linear branches. By the statistical stability of such maps we hence have \eqref{18}.

By \eqref{eq:dL1} and since we saw that $d\mathcal{L}^{(0)}_{\delta,h}\eta=T_*\Lambda_{\delta,h}\eta$ we have, for each $\eta\in  W^{1,1}$ and $f$ continuous,
\[
\begin{split}
&\int_{\mathbb{T}} (d\mathcal{L}^{(\epsilon)}_{\delta,h_\epsilon}\eta-d\mathcal{L}^{(0)}_{\delta,h}\eta)\cdot f=\\
&=\int_{\mathbb{T}} \big([T_{\epsilon}]_*\Lambda_{\delta ,h_\epsilon }(\eta) +\delta m(\eta H)[T_{\epsilon}]_*\Lambda_{\delta ,h_\epsilon }(h_\epsilon')-T_*\Lambda_{\delta ,h }(\eta)\big)f\\
&=\int_{\mathbb{T}} [\Lambda_{\delta ,h_\epsilon }(\eta)+\delta m(\eta H)\Lambda_{\delta ,h_\epsilon }(h_\epsilon')]\cdot f\circ T_\epsilon-\Lambda_{\delta ,h }(\eta)\cdot f\circ T_.
\end{split}
\]
Let us call $\hat \eta_\epsilon=\Lambda_{\delta ,h_\epsilon}(\eta)+\delta m(\eta H)\Lambda_{\delta ,h_\epsilon }(h_\epsilon')$. By the above equation we can write
\begin{equation}\label{eq:TOP}
\int_{\mathbb{T}} (d\mathcal{L}^{(\epsilon)}_{\delta,h_\epsilon}\eta-d\mathcal{L}^{(0)}_{\delta,h}\eta)\cdot f=\int_{\mathbb{T}} \hat\eta_\epsilon[ f\circ T_\epsilon-f\circ T]+\int_{\mathbb{ T}} [\hat\eta_\epsilon -\Lambda_{\delta ,h }\eta] f\circ T.
\end{equation}
Setting $F_\epsilon(x)=\frac 1k\int_{T(x)}^{T_\epsilon(x)} f(t)dt$ we have
\[
f\circ T(x)=\frac 1k T'_\epsilon(x)f\circ T_\epsilon-F_\epsilon'(x).
\]
Therefore,
\[
\int_{\mathbb{T}} \hat \eta_\epsilon\cdot (f\circ T-f\circ T_\epsilon)=\int_{\mathbb{T}}\hat \eta_\epsilon\left(\frac 1k T'_\epsilon(x)-1\right)f\circ T_\epsilon -\int_{\mathbb{T}} \hat \eta_\epsilon F'_\epsilon.
\]
Since, for each $x\in \mathbb{T}$,
\[
|F_\epsilon'(x)|\le k^{-1} \|T_\epsilon-T\|_{\mathcal C^1}\|f\|_{L^{\infty}}\le k^{-1} \epsilon \|f\|_{L^{\infty}}, 
\]
and
\[
|1-k^{-1}T'_\epsilon(x)|\le k^{-1}\epsilon,
\]
applying the above to \eqref{eq:TOP} and noting that, by \eqref{18}, $\|\hat{\eta}_\epsilon-\Lambda_{\delta,h}\eta\|_{W^{1,1}}\le C\epsilon$, we obtain
\[
\int_{\mathbb{T}} |(d\mathcal{L}^{(\epsilon)}_{\delta,h_\epsilon}\eta-d\mathcal{L}^{(0)}_{\delta,h}\eta) f|\le k^{-1} \epsilon \|\hat \eta_\epsilon\|_{W^{1,1}} \|f\|_{L^{\infty}}+C k^{-1} \epsilon \|\hat \eta_\epsilon\|_{L^1}\|f\|_{L^{\infty}}.
\]
Item $(ii)$ follows since, by \eqref{eq:Lambda} ,$\|\hat \eta_\epsilon\|_{W^{1,1}}\le C\|\eta\|_{W^{1,1}}$.\\
By \cite[Theorem 1]{KL} we conclude that, for $\epsilon_0$ sufficiently small, the operator $d%
\mathcal{L}^{(\epsilon)}_{\delta,h_\epsilon}$ restricted to $V_{s}$ has
spectral radius strictly smaller than one and Proposition \ref{fond} implies that the family of STO for the perturbed system enjoy local exponential convergence to equilibrium.
\end{proof}

\begin{remark}
    Note that, in the above argument, we used the specific form of the uncoupled dynamics (linear expanding) only to ensure that the spectral radius of the associated differential of the STO contracts the set of zero average measures. Therefore, if one can prove this independently for a more general expanding map, the argument for the perturbation remains valid as long as the distortion is uniformly bounded, as it does not critically rely on the linearity of the uncoupled map.
\end{remark}

\subsection{An example with random coupling
}\label{ex2}

In this section, we describe a class of examples involving stochastic coupling, for which we can prove, using Proposition \ref{fond}, that while there is only one invariant measure in $L^1$
  for the associated self-consistent transfer operator in the weak coupling regime, multiple locally stable measures exist in a stronger coupling regime.

The example represents stochastic mean-field dynamics, where each subsystem has a certain probability of either switching to a specific distribution of states (e.g., approaching the average or dominant state) or following its internal dynamics. The probability of switching depends on the global state of the system; for example, it may depend on how dominant the average state is. This is measured by the ``barycenter" of the distribution of the systems in the phase space.

The phase space of the systems will be $\mathbb{T}$ and, as
uncoupled dynamics, we consider $T:\mathbb{T}\rightarrow \mathbb{T}$ to be a $C^{4}$
expanding map of the circle.\\
Since in the following we are interested in measures absolutely continuous with respect to the Lebesgue measure , we will let the transfer operator $L_{T}$ \ associated to $T$ acting
on densities $L_{T}:L^{1}(\mathbb{T)}\rightarrow L^{1}(\mathbb{T)}$. \ Let
us define more precisely all the tools in the construction and give a motivation for choice of the coupling. In the
following we will consider $L^{1},W^{1,1},W^{2,1}$ as weak, strong and
strongest space respectively, hence for example $P_{w}$ will be the set of probability
measures in $L^{1}(\mathbb{T)}$.

Let us suppose that at some time the global state of the system is
represented by a certain measure $\mu $ having density $\phi \in L^{1}(%
\mathbb{T)}$ and at the next iterate each subsystem has a certain
probability to follow the dynamics $T$ or move to a certain distribution $%
\psi _{\phi }\in P_{ss}$ depending on $\phi .$ To formalize this we define
the weight $W_{\phi }$ of the barycenter of distribution $\phi $ as%
\begin{equation*}
W_{\phi }=|\int_{\mathbb{T}}e^{ix}d\mu |^{2}
\end{equation*}%
and the probability for a subsystem to switch to the distribution $\psi
_{\phi }$ will be proportional to $W_{\phi }$ (see $(\ref{traskip})$).

About the function $\phi \rightarrow \psi _{\phi }$ we suppose that:

\begin{itemize}
\item $\sup_{\phi \in P_{w}}\|\psi _{\phi }\|_{W^{2,1}}:=M\in \mathbb{R}$ and


\item $\phi \to \psi _{\phi }$ is Lipschitz in $\phi $: $\exists C\geq 0$ such
that $\|\psi _{\phi +g}-\psi _{\phi }\|_{W^{1,1}}\leq C\delta \|g\|_{L^{1}}$.
\end{itemize}

For example $\psi _{\phi }$ can be a Gaussian centered at the barycenter of
\ the distribution $\phi .$ \ When $W_{\phi }\neq 0$ the barycenter $\bar{%
\phi}$ of $\phi $ on $\mathbb{T}$ can be defined by%
\begin{equation*}
\bar{\phi}=\arg (\int_{\mathbb{T}} e^{ix}d\mu ).
\end{equation*}%
\par
In this example one we can choose $\sigma >0$ and define 
\begin{equation}\label{pzi}
\psi _{\phi }=\frac{1}{\sqrt{2\pi \sigma ^{2}}}\sum_{k\in \mathbb{Z}%
}e^{\frac{-(x-\bar{\phi}-2\pi k)^{2}}{2\sigma ^{2}}}.
\end{equation}%

Considering some $\delta \geq 0$, as a coupling strength parameter, the
transfer operator $L_{\delta ,\phi }:L^{1}\rightarrow L^{1}$ associated to the coupled system can be defined as%
\begin{equation}
L_{\delta ,\phi }(f)= 
\begin{cases}
(1-\delta W_{\phi })L_{T}(f)+[\delta W_{\phi }\int f~dx]~\psi _{\phi
}\quad  &\text{when}\quad\delta W_{\phi }\leq 1, \\ \\
\psi _{\phi }\int f~dx\quad  &\text{when}\quad\delta W_{\phi }>1.%
\end{cases}%
  \label{traskip}
\end{equation}

The associated self consistend transfer operator is then defined as%
\begin{equation*}
\mathcal{L}_{\delta }(f)=L_{\delta ,f }(f).
\end{equation*}

In this framework we can prove the following propositions, one describing
the global behavior of the self-consistent transfer operator in the weak
coupling regime, and then in the strong coupling one.

\begin{lemma}
\label{lemmassumpt}Under the above assumptions,
suppose that $T:\mathbb{%
T\rightarrow T}$ has an invariant density $h_{0}\in W^{3,1}$ such that%
\footnote{%
The doubling map and many other maps having a central symmetry have this
property.} $W_{h_{0}}=0$. Consider $\tilde{\phi}\in
L^{1}$ such that $\delta W_{\tilde{\phi}}>1$.
Let us suppose furthermore that  
$\phi \rightarrow \psi _{\phi }$ is \ differentiable \ at  $\phi= \tilde{\phi}$:  there is a bounded linear  $\dot{\psi}_{\tilde{\phi} }:L^{1}\rightarrow W^{2,1}$
such that 
\begin{equation*}
\lim_{g\in V_{L^{1}}\|g\|_{L^{1}}\rightarrow 0}\frac{\|\psi _{\tilde{\phi}  +g}-\psi
_{\tilde{\phi}  }-\dot{\psi}_{\tilde{\phi} }(g)\|_{W^{1,1}}}{\|g\|_{L^{1}}}=0 
\end{equation*} and
$\|\dot{\psi}_{\tilde{\phi} }\|_{w\to s}<1.$ \footnote{A simple example satisfying this  can be constructed by considering in the example given at \eqref{pzi}, $\sigma$ large enough, as for $\sigma \to \infty$, $\psi_\phi$ tends to a, the uniform distribution. }
Then $h_{0}$ and $\psi_{\tilde{\phi}}$
are fixed points of $\mathcal{L}_{\delta }$, and the family of operators $%
L_{\delta ,\phi }$ satisfies the assumptions $\bf{(a)},...,\bf{(f)}$ of section \ref{S2}
for both these fixed points (i.e. considering the assumptions $\bf{(a)},...,\bf{(f)}$
both for $h=h_{0}$ and $h=\tilde{\phi}$).
\end{lemma}

\begin{proof}
The assumptions $\bf{(a)},\bf{(b)}$ are trivial, the assumption $\bf{(c)}$ comes from the fact
that when $\delta W_{\phi }\leq 1$%
\begin{eqnarray*}
\|L_{\delta ,\phi }(f)\|_{s} &=&\|(1-\delta W_{\phi })L_{T}(f)+[\delta
W_{\phi }\int f~dx]~\psi _{\phi }\|_{s} \\
&\leq &(1-\delta W_{\phi })[\lambda \|f\|_{s}+B\|f\|_{w}]+[\delta W_{\phi
}]~M|\int f~dx| \\
&\leq &\lambda \|f\|_{s}+[B+M]\|f\|_{w}.
\end{eqnarray*}

While, when $\delta W_{\phi }>1$%
\begin{equation*}
\|L_{\delta ,\phi }(f)\|_s=\|\psi _{\phi }\int f~dx\|_{s}\leq M\|f\|_{w}.
\end{equation*}

The assumption $\bf{(d)}$ \ comes from the fact that since $W_{h_0}=0$ then $%
L_{\delta ,h_{0}}$ is just the transfer operator of an expanding map. While
in the other case $L_{\delta ,\psi _{\phi }}(f)=\psi _{\phi }\int f~dx=0$
when $\int f~dx=0$.

To verify assumption $\bf{(e)}$ we consider $\phi _{1},\phi _{2}\in
P_{w}$ . We remark that 
\begin{eqnarray*}
|W_{\phi _{1}}-W_{\phi _{2}}| &=&||\int_{\mathbb{T}}e^{ix}~\phi
_{1}(x)~dx|^2-|\int_{\mathbb{T}}e^{ix}~\phi _{2}(x)~dx|^2| \\
&\leq &4\pi \|\phi _{1}-\phi _{2}\|_{L^{1}}.
\end{eqnarray*}%
Then%
\begin{eqnarray*}
\|L_{\delta ,\phi _{1}}(f)-L_{\delta ,\phi _{2}}(f)\|_{w} &=&\|L_{\delta
,\phi _{1}}(f)-L_{\delta ,\phi _{2}}(f)\|_{L^{1}} \\
&\leq &\|(1-\delta W_{\phi _{1}})L_{T}(f)+[\delta W_{\phi _{1}}\int
f~dx]~\psi _{\phi _{1}} \\
&&-(1-\delta W_{\phi _{2}})L_{T}(f)+[\delta W_{\phi _{2}}\int f~dx]~\psi
_{\phi _{2}}\|_{w} \\
&\leq &4\pi \delta \|\phi _{1}-\phi _{2}\|_{L^{1}}\|f\|_{L^{1}}+4M\pi \delta
\|\phi _{1}-\phi _{2}\|_{L^{1}}\|f\|_{L^{1}} \\
&&+\|f\|_{L^{1}}C\delta \|\phi _{1}-\phi _{2}\|_{L^{1}}.
\end{eqnarray*}

Similar estimates can be achieved for $\|L_{\delta ,\phi _{1}}(f)-L_{\delta
,\phi _{2}}(f)\|_{s}$ using the fact that $L_{T}:B_{s}\rightarrow B_{s}$ is
bounded, completing the verification of $\bf{(e)}.$

Now we verify the differentiability assumption $\bf{(f)}.$ Let us check the
definition first at $h_{0}$, then at $\psi _{\tilde{\phi} }.$

We have to verify that there is $\partial L_{\delta ,h_{0}}$such that 
\begin{equation}
\lim_{g\in V_{w}\Vert g\Vert _{w}\rightarrow 0}\frac{\Vert L_{\delta
,h_{0}+g}(h)-L_{\delta ,h_{0}}(h)-\partial L_{\delta ,h_{0}}(g)\Vert _{s}}{%
\Vert g\Vert _{w}}=0.
\end{equation}

We first compute $\partial L_{\delta ,h_{0}}$ proving that $L_{\delta
,h_{0}+g}(h)-L_{\delta ,h_{0}}(h)=\partial L_{\delta ,h_{0}}(g)+o(g)$ where
\ $\frac{\|o(g)\|_{s}}{\|g\|_{w}}\rightarrow 0$ for $g\rightarrow 0$ in $%
V_{w}$.

First, we remark that, by the assumptions, $\delta W_{h_{0}+g}<1$  for each $g$ such that $\|g\|_{w}$ is small enough, furthermore $W_{h_{0}+g}=o(||g||_{w})$ as $||g||_{w}\to 0$.

Thus in this case
\begin{eqnarray*}
L_{\delta ,h_{0}+g}(h_{0})-L_{\delta ,h_{0}}(h_{0}) &=&(1-\delta
W_{h_{0}+g})L_{T}(h_{0})+\delta W_{h_{0}+g}\psi _{h_{0}+g}-h_{0} \\
&=&-\delta W_{h_{0}+g}h_{0}+\delta W_{h_{0}+g}\psi _{h_{0}+g} \\
&=& o_{w\to s}(g)
\end{eqnarray*}
as $||g||_{w}\to 0$, where $\dfrac{\|o_{w\to s}(g)\|_s}{\|g\|_{w}}\to  0$ \  uniformly on $g$ as $\|g\|_{w}\rightarrow 0$.
By this, the differential $\partial L_{\delta ,h_{0}}=0$.

Now we apply the assumption $~\delta W_{\tilde{\phi}}>1$ and compute the
differential at $\psi _{\tilde{\phi}}$. We remark that for $g\in V_{w}$
small enough \ $\delta W_{\psi _{\tilde{\phi}}+g}>1$, then%
\begin{eqnarray*}
L_{\delta ,\psi _{\tilde{\phi}}+g}(\psi _{\tilde{\phi}})-L_{\delta ,\psi _{%
\tilde{\phi}}}(\psi _{\tilde{\phi}}) &=&[\psi _{\tilde{\phi}+g}-\psi _{%
\tilde{\phi}}] \\
&=&\dot{\psi}_{\tilde{\phi}}(g)+o_{w\to s}(g)
\end{eqnarray*}%
were as above $\frac{\|o_{w\to s}(g)\|_{s}}{\|g\|_{w}}\rightarrow 0$ for $g\rightarrow 0$ in $%
V_{w}$. Thus, $\partial L_{\delta,{ \psi_{\bar \phi}}}(g)=\dot{\psi}_{\tilde{\phi}}(g).$
\end{proof}

The following proposition describes the behavior of the system in the weak
coupling behavior.
\begin{proposition}
Under the assumptions of Lemma \ref{lemmassumpt}, there is $0<\bar{\delta}<1$
\ such that for each $\delta \leq \bar{\delta}$ the self-consistent operator $%
\mathcal{L}_{\delta }$ defined above has only $h_{0}$ as invariant
probability density in $L^{1}.$
\end{proposition}

\begin{proof}
We prove the uniqueness of the invariant measure $h_{0}$ in $L^{1}$ by using
the criterion established in Theorem 4 of \cite{G}.

In our setting, such statement requires as assumptions that once fixed some $%
\overline{\delta }>0$  we can verify that:

\begin{itemize}
\item for each $0\leq \delta <\overline{\delta }$ and $\phi \in P_{w}$ the
operators $L_{\delta ,\phi }$ have unique invariant probability density $%
h_{\delta ,\phi }\in P_{w}$ such that $\sup_{0\leq \delta <\overline{\delta }%
,\phi \in P_{w}}\{\|h_{\delta ,\phi }\|_{s}\}<\infty $, and

\item the operators $L_{\delta ,\phi }$ satisfy $(\ref{e1.5})$

\item there is $K_{2}\geq 1$ such that $\forall f_{1},f_{2}\in P_{w},0\leq
\delta <\overline{\delta }$ 
\begin{equation}
\|h_{\delta ,f_{1}}-h_{\delta ,f_{2}}\|_{L^{1}}\leq \delta
K_{2}\|f_{1}-f_{2}\|_{L^{1}}.  \label{lipp}
\end{equation}
\end{itemize}

Under these assumptions the statement establishes that for all $0\leq \delta
\leq \min (\overline{\delta },\frac{1}{K_{2}})$, there is a unique $%
h_{\delta }\in P_{w}$ \ such that%
\begin{equation*}
\mathcal{L}_{\delta }(h_{\delta })=h_{\delta }.
\end{equation*}

In Lemma \ref{lemmassumpt} we verified the first two items in the above
bullet list. Indeed the first item follows from the uniform Lasota Yorke
inequality $\bf{}(c).$ The second item is part of assumption $\bf{}(e)$. Hence, to establish
the result, it remains to verify $(\ref{lipp})$. This is standard from
the uniform spectral gap of the operators $L_{\delta ,\phi }$ when $\delta $
is small. It can be formalized similar to what is done in Section 7.4 of 
\cite{Gal}. We sketch the reasoning: the operators $L_{\delta ,\phi }$ as $%
\phi $ varies in $P_{w}$ satisfy a uniform Lasota-Yorke inequality and
furthermore there is $C\geq 0$ such that 
\begin{eqnarray}
\|(L_{\delta ,\phi _{1}}-L_{\delta ,\phi _{2}})\phi \|_{w} &\leq &\delta
C\|\phi _{1}-\phi _{2}\|_{w}\|\phi \|_{s}, \label{lop} \\
\|(L_{\delta ,\phi _{1}}-L_{\delta ,\phi _{2}})h_{\delta ,\phi _{1}}\|_{s}
&\leq &\delta C\|\phi _{1}-\phi _{2}\|_{w}.
\end{eqnarray}

This implies that there is $M_{2}\geq 0$ such that for each $\phi _{2}\in
P_{w}$, $\|(I-L_{\delta ,\phi _{2}})^{-1}\|_{V_{s}\rightarrow
V_{s}}\leq M_{2}$.

Thus we have:%
\begin{eqnarray*}
(I-L_{\delta ,\phi _{2}})(h_{\delta ,\phi _{2}}-h_{\delta ,\phi _{1}})
&=&h_{\delta ,\phi _{2}}-L_{\delta ,\phi _{2}}h_{\delta ,\phi
_{2}}-h_{\delta ,\phi _{1}}+L_{\delta ,\phi _{2}}h_{\delta ,\phi _{1}} \\
&=&(L_{\delta ,\phi _{1}}-L_{\delta ,\phi _{2}})h_{\delta ,\phi _{1}}
\end{eqnarray*}%
and we get%
\begin{equation*}
(h_{\delta ,\phi _{2}}-h_{\delta ,\phi _{1}})=(I-L_{\delta ,\phi
_{2}})^{-1}(L_{\delta ,\phi _{1}}-L_{\delta ,\phi _{2}})h_{\delta ,\phi
_{1}}.
\end{equation*}
Then%
\begin{equation}
\|h_{\delta ,\phi _{2}}-h_{\delta ,\phi _{1}}\|_{w}\leq \|h_{\delta ,\phi
_{2}}-h_{\delta ,\phi _{1}}\|_{s}\leq \delta C\|\phi _{1}-\phi _{2}\|_{w}.
\end{equation}%
and the statement is proved.
\end{proof}

The following proposition describes the behavior of the system in the strong
coupling behavior.

\begin{proposition}
Under the assumptions of Lemma \ref{lemmassumpt}, supposing furthermore that 
$\|\dot{\psi}_{\tilde{\phi}}\|_{w\rightarrow ss}<1$ then both $h_{0}$ and $%
\psi _{\tilde{\phi}}$ \ are stable fixed points of $\mathcal{L}_{\delta }$
having local exponential convergence to equilibrium. There are $r,C,\lambda
>0$ such that for each $g\in W^{1,1}$ with $\int gdx=0$ and $%
\|g\|_{W^{1,1}}\leq r$ we have%
\begin{equation*}
\|\mathcal{L}_{\delta }^{n}(h_{0}+g)-h_{0}\|_{W^{1,1}}\leq Ce^{-\lambda n}
\end{equation*}%
and%
\begin{equation*}
\|\mathcal{L}_{\delta }^{n}(\psi _{\tilde{\phi}}+g)-\psi _{\tilde{\phi}%
}\|_{W^{1,1}}\leq Ce^{-\lambda n}.
\end{equation*}
\end{proposition}

\begin{proof}
We prove the statement by applying Proposition \ref{fond}, choosing $%
L^{1},W^{1,1},W^{2,1}$ as weak, strong and strongest space. In Lemma \ref%
{lemmassumpt} we proved that the system satisfies the assumptions $\bf{(a)}.,...,\bf{(f)}$ and we computed $\partial L_{\delta
,h_{0}}=0$, by this item $1)$ of Proposition \ref{fond} is automatically verified. We then have that the differential at $h_{0}$ is the function associating to $g\in
V_{s}$ \ the value%
\begin{eqnarray*}
d\mathcal{L}_{\delta ,h_{0}}(g) &=&L_{\delta ,h_{0}}(g)+\partial L_{\delta
,h_{0}}(g) \\
&=&L_{\delta ,h_{0}}(g).
\end{eqnarray*}

By this we know that one iterate of the differential will contract the zero average space and then we can  apply Proposition \ref{fond}.
obtaining that  $h_{0}$ \ is a stable fixed point
of $\mathcal{L}_{\delta }$ and we have local exponential contraction,
proving the first part of the the main claim.\\
On the other hand considering the other fixed point $\psi _{\tilde{\phi}}$, by \eqref{traskip} %
\begin{eqnarray*}
d\mathcal{L}_{\delta ,\psi _{\tilde{\phi}}}(g) &=&L_{\delta ,\psi _{\tilde{%
\phi}}}(g)+\dot{\psi}_{\tilde{\phi}}(g) \\
&=&[0+\dot{\psi}_{\tilde{\phi}}(g)] \\
&=&\dot{\psi}_{\tilde{\phi}}(g),
\end{eqnarray*}
since we suppose $\|\dot{\psi}_{\tilde{\phi}}\|_{w\rightarrow ss}<1$ then  $%
\partial L_{\delta ,h_{0}}:V_w\rightarrow B_{ss}$ is bounded,
furthermore $d\mathcal{L}_{\delta ,\psi _{\tilde{\phi}}}$ is a contraction.
This shows that $\psi_{\tilde{\phi}}$ is  a stable fixed point and the second part of the claim.
\end{proof}

\begin{remark} We remark that when $\delta W_{\tilde{\phi}}>1$ and $\sigma$ (see \eqref{pzi}) is small enough the example given in \eqref{pzi} does not satisfy $\|\dot{\psi}_{\tilde{\phi}}\|_{w\to ss}<1$ but still $\delta W_{\psi_{\tilde{\phi}}}>1$. In this case since the barycenter of $\tilde{\phi}$ is the same of ${\psi}_{\tilde{\phi}}$, then ${\psi}_{\tilde{\phi}}={\psi}_{\psi_{\tilde{\phi}}},$  and ${\psi}_{\tilde{\phi}}$ is a (indifferent) fixed point for the system.
\end{remark}

\appendix

\section{Exponential loss of memory for sequential composition of
operators}

In this section, we present a relatively simple and general argument establishing exponential loss of memory for a sequential composition of Markov operators converging to a limit. Since the approach is general, we will work within an abstract framework, stating a result that holds for a sequence of Markov operators acting on suitable spaces of measures.

Let $B_{w}$ and $B_{s}$ be normed vector subspaces of the set of signed measures on $%
X $. Suppose $(B_{s},\|~\|_{s})\subseteq $ $(B_{w},\|~\|_{w})$ and $%
\|~\|_{s}\geq \|~\|_{w}$. Let us consider a sequence of linear Markov
operators $\{L_{i}\}_{i\in \mathbb{N}}:B_{s}\rightarrow B_{s}.$ \ We will
suppose furthermore that the following assumptions are satisfied by the $%
L_{i}$:

\begin{itemize}
\item[$(ML1)$] The operators $L_{i}$ satisfy a common Lasota-Yorke inequality.
There are constants $A,B,M,\lambda _{1}\geq 0$ with $\lambda _{1}<1\leq B $ such
that for all $f\in B_{s},$ $\mu \in P_{w},$ $i,n\in \mathbb{N}$%
\begin{equation} \label{1}
\begin{split}
\|L_{i+n}\circ ...\circ L_{i+1}(f)\|_{w}&\leq M \|f\|_{w} \\ 
\|L_{i+n}\circ ...\circ L_{i+1}(f)\|_{s}&\leq A\lambda
_{1}^{n}\|f\|_{s}+B\|f\|_{w}.%
\end{split} 
\end{equation}

\item[$(ML2)$] There is $\epsilon >0$ and a Markov operator $
L_{0}:B_{s}\rightarrow B_{s}$ such that : $\forall j>0$ 
\begin{equation}
\|(L_{0}-L_{j})\|_{B_{s}\rightarrow B_{w}}\leq \epsilon .
\end{equation}

\item[$(ML3)$] There exists a sequence of real and positive numbers $a_{n}\geq 0$ with $a_{n}\rightarrow 0$ such that
for all $n\in \mathbb{N}$ and $v\in V_{s}$
\begin{equation} \label{3}
\|L_{0}^{n}(v)\|_{w}\leq a_{n}\|v\|_{s} ,
\end{equation}%
where%
\begin{equation*}
V_{s}=\{\mu \in B_{s}|\mu (X)=0\}.
\end{equation*}
\end{itemize}

We remark that assumption $(ML1)$ implies that the family of operators $%
L_{i}$ is uniformly bounded when acting on $B_{s}$ and on $B_{w}.$

The following lemma is an estimate for the distance of the sequential
composition of operators from the iterates of $L_0$.

\begin{lemma}
\label{XXX}Let $\delta \geq 0$ and let $L^{(j+1,j+n)}:=L_{j+n}\circ ...\circ
L_{j+1}$ be a sequential composition of operators $\{L_{i}\}_{i\in \mathbb{N}%
}$ \ satisfying $(ML1),(ML2)$ and $(ML3)$. Let $L_{0}$ be such that $%
\|L_{i}-L_{0}\|_{s\rightarrow w}\leq \delta $ for each $i\in \mathbb{N}$.\ Then there is $C\geq 0$ such
that, for each $ g\in B_{s}$ and $j,n\geq 1$ it holds%
\begin{equation}
\|L^{(j,j+n-1)}g-L_{0}^{n}g\|_{w}\leq M^n \delta (C\|g\|_{s}+n\frac{B}{1-\lambda 
}\|g\|_{w}),  \label{2}
\end{equation}%
where $B$ is the constant in \eqref{1}.
\end{lemma}

\begin{proof}
We proceed by induction on $n\ge 1$. By the assumptions we get%
\begin{equation*}
\|L_{0}g-L_{j}g\|_{w}\leq \delta \|g\|_{s},
\end{equation*}
therefore the case $n=1$ of $(\ref{2})$ is trivial. \ Let us now suppose
inductively that, for each $n \ge 2$, there exists $C_n>0$ such that%
\begin{equation*}
\|L^{(j,j+n-2)}g-L_{0}^{n-1}g\|_{w}\leq M^{n-1}\delta \left(C_{n-1}\|g\|_{s}+(n-1)\frac{B%
}{1-\lambda _{1}}\|g\|_{w}\right),
\end{equation*}%
for some $0<\lambda_1<1$. Then,%
\begin{equation*}
\begin{split}
&\|L_{j+n-1}L^{(j,j+n-2)}g-L_{0}^{n}g\|_{w} \le\\
&\leq \|L_{j+n-1}L^{(j,j+n-2)}g-L_{j+n-1}L_{0}^{n-1}g+L_{j+n-1}L_{0}^{n-1}g-L_{0}^{n}g\|_{w}
\\
&\leq \|L_{j+n-1}L^{(j,j+n-2)}g-L_{j+n-1}L_{0}^{n-1}g\|_{w}+\|L_{j+n-1}L_{0}^{n-1}g-L_{0}^{n}g\|_{w} \\
&\leq M^n\delta \left(C_{n-1}\|g\|_{s}+(n-1)\frac{B}{1-\lambda _{1}}\|g\|_{w}\right) +\|[L_{j+n-1}-L_{0}](L_{0}^{n-1}g)\|_{w} \\
&\leq M^n \delta \left( C_{n-1}\|g\|_{s}+(n-1)\frac{B}{1-\lambda _{1}}%
\|g\|_{w}\right)+\delta \|L_{0}^{n-1}g\|_{s} \\
&\leq M^n\delta \left(C_{n-1}\|g\|_{s}+(n-1)\frac{B}{1-\lambda _{1}}\|g\|_{w}\right) +\delta \left(A\lambda _{1}^{n-1}\|g\|_{s}+\frac{B}{1-\lambda _{1}}\|g\|_{w}\right) \\
&\leq M^n\delta \bigg[ (C_{n-1}+A\lambda _{1}^{n-1})\|g\|_{s}+n\frac{B}{%
1-\lambda _{1}}\|g\|_{w}\bigg].
\end{split}
\end{equation*}

Since $\lambda_1<1$, iterating the above inequality yields a bound of $C_{n-1}$ by a geometric series, from which the statement follows.
\end{proof}

\begin{lemma}
\label{losmem} Let $L_{i}$ be a sequence of operators satisfying $(ML1),(ML2)$ and $(ML3)$. Then the sequence $L_{i}$ has
strong exponential loss of memory in the following sense. There are $C,\lambda > 0$ such that for each $\epsilon>0$  small enough \footnote{\label{notaunif}We stress that the coefficients $C,\lambda $ depend on $ \lambda_1 , A,B,M$ and the sequence $a_n$ in the assumptions $(ML1)...(ML3)$ and not on $\epsilon$ and the chosen sequence of operators $L_i$ satisfying the assumptions $(ML1)...(ML3)$ themselves, provided $\epsilon$ is chosen small enough. More precisely, given $ \lambda_1 , A,B,M\geq 0$  there are $\epsilon_0, C,\lambda$ such that for each $0\geq \epsilon \geq \epsilon_0$ and a family of operators satisfying $(ML1),...,(ML3)$ with paremeters  $\lambda_1 , A,B,\epsilon$ we have the loss of memory with paremeters $C,\lambda$. 
}
and for each $j,n\in \mathbb{N}$ and $g\in V_{s}$, 
\begin{equation*}
\|L^{(j,j+n-1)}g\|_{s}\leq Ce^{-\lambda n}\|g\|_{s}.
\end{equation*}
\end{lemma}

\begin{proof}
By $(ML1)$, for each $ j,i\geq 1$ and $g\in
B_{s}$,
\begin{equation*}
\|L^{(j,j+i)}(g)\|_{s}\leq \left(\frac{B}{1-\lambda _{1}}+1\right)\|g\|_{s}.
\end{equation*}%
Now let us consider $N_{0} \in \mathbb{N}$ such that 
\[
A\lambda _{1}^{N_{0}}\leq \frac{1}{100\left(%
\dfrac{B}{1-\lambda _{1}}+1\right)}.
\]
By $(ML3)$, there is $N_{2}\in \mathbb{N}$ such that, for each $
i\geq N_{2} $ and $g\in V_{s}$ ,
\begin{equation*}
\|{L_{0}}^{N_{2}}g\|_{w}\leq \frac{1}{100B}\|g\|_{s}.
\end{equation*}%
Let $\bar N:=\max (N_{0},N_{2})$ and let $\epsilon \leq \frac{(1-\lambda _{1})}{%
100\bar N M^{\bar N}B({C+B})}$. Then%
\begin{equation*}
\|L_{i}-L_{0}\|_{s\rightarrow w}\leq \frac{(1-\lambda _{1})}{100[\bar N M^{\bar N}]B({C+B})}.
\end{equation*}%
By $(\ref{2})$, for each $j\geq 1,i\leq \bar N$ ,
\begin{equation*}
\begin{split}
\|L^{(j,j+i-1)}g-{L_{0}}^{i}g\|_{w}& \leq \frac{(1-\lambda _{1})}{100\bar NM^{\bar N}B(C+B)}M^i
\left(C\|g\|_{s}+i\frac{B\|g\|_{w}}{(1-\lambda _{1})}\right) \\
& \leq \frac{iM^i}{100\bar NM^{\bar N}B}\|g\|_{s} \\
& \leq \frac{1}{100 B}\|g\|_{s}.
\end{split}%
\end{equation*}%
Hence,
\begin{equation}
\begin{split}
\|L^{(j,j+\bar N-1)}g\|_{w}& \leq \|{L_{0}}^{\bar N}g\|_{w}+\frac{1}{100B}\|g\|_{s} \\
& \leq \frac{1}{100B}\|g\|_{s}+\frac{1}{100B}\|g\|_{s}.
\end{split}%
\end{equation}
Applying the Lasota-Yorke inequality \eqref{1} we get, for any $j\geq 1$,
\begin{equation}
\begin{split}
\|L^{(j,j+2\bar N-1)}g\Vert _{s}& \leq A\lambda
_{1}^{-\bar N}\|L^{(j,j+\bar N-1)}g\|_{s}+B\Vert L^{(j,j+\bar N-1)}g\Vert _{w} \\
& \leq \frac{1}{100}\|g\|_{s}+\frac{B}{100B}\|g\|_{s}+\frac{B}{100B}%
\|g\|_{s} \\
& \leq \frac{3}{100}\|g\|_{s}
\end{split}%
\end{equation}%
and 
\begin{equation*}
\|L^{(j,j+2k \bar N-1)}g\Vert _{s}\leq \left(\frac{3}{100}\right)^{k}\|g\|_{s}
\end{equation*}%
for each $j\geq 1$, $k\geq 1$ and $g\in V_{s}$, establishing the result.
\end{proof}

\

\noindent \textbf{Acknowledgements}

 S.G. acknowledges the MIUR Excellence Department Project awarded to the
Department of Mathematics, University of Pisa, CUP I57G22000700001. S.G. was
partially supported by the research project "Stochastic properties of
dynamical systems" (PRIN 2022NTKXCX) of the Italian Ministry of Education
and Research. M.T. acknowledges Marie Slodowska-Curie Actions: "Ergodic
Theory of Complex Systems", project no. 843880. The authors acknowledge the UMI Group “DinAmicI” ({www.dinamici.org}) and the INdAM group GNFM.
The authors wish to thank B. Fernandez for useful discussions during the preparation of the work.
\\
\textbf{Declarations}

The authors have no relevant financial or non-financial interests to disclose. 
Data sharing not applicable to this article as no datasets were generated or analysed during the current study.















\end{document}